\documentclass{article}
\usepackage[utf8]{inputenc}
\usepackage{srcltx}
\usepackage{amssymb,amsmath,amsfonts,amsthm}
\usepackage{cite}
\usepackage{verbatim}

\usepackage[english]{babel}
\usepackage{amsmath}
\usepackage{longtable}
\usepackage{tikz}
\usepackage{amssymb}
\usepackage{amsfonts}
\usepackage{textcomp}

\newtheorem{theorem}{Theorem}[section]

\theoremstyle{definition}

\theoremstyle{remark}

\numberwithin{equation}{section}

\DeclareMathOperator{\SL}{SL}

 \DeclareMathOperator{\PSp}{PSp}

\DeclareMathOperator{\sgn}{sgn}

\newcommand{\F}{\mathbb{F}}

\newtheorem{lem}{{\bfseries Lemma}}

\newtheorem{prop}{{\bfseries Proposition}}
\newtheorem*{lem*}{{\bfseries Lemma}}
\newcommand*{\No}{\textnumero}

\title{On splitting of the normalizer of a maximal torus in $E_6(q)$}

\author{Alexey Galt and Alexey Staroletov\footnote{This research was supported by the Russian Science Foundation (project no. 14-21-00065).}}

\date{\vspace{-5ex}}
\begin{document}
\newcommand{\Addresses}{{
  \bigskip
  \footnotesize

  A.~Galt, \textsc{Sobolev Institute of Mathematics, Novosibirsk, Russia;}\par\nopagebreak
  \textit{E-mail address: } \texttt{galt84@gmail.com}

  \medskip

  A.~Staroletov, \textsc{Sobolev Institute of Mathematics, Novosibirsk, Russia;}\par\nopagebreak
  \textit{E-mail address: } \texttt{staroletov@math.nsc.ru}
}}


\maketitle
\begin{abstract}
Let $G$ be a finite group of Lie type $E_6$ over $\F_q$ (adjoint or simply connected) and $W$ be the Weyl group of $G$. We describe maximal tori $T$ such that $T$ has a complement in its algebraic normalizer $N(G,T)$. 
It is well known that for each maximal torus $T$ of $G$ there exists an element $w\in W$ such that $N(G,T)/T\simeq C_W(w)$. When $T$ does not have a complement isomorphic to $C_W(w)$, we show that $w$ has a lift in $N(G,T)$ of the same order.
\end{abstract}

\section{Introduction}
Finite groups of Lie type arise from linear algebraic groups as sets of fixed points of a Steinberg endomorphism.
Let $\overline{G}$ be a simple connected linear algebraic group over an algebraically closed field $\overline{\F}_p$ of positive characteristic $p$. Consider a Steinberg endomorphism $\sigma$ and a maximal $\sigma$-invariant torus $\overline{T}$  of $\overline{G}$. It is well known that all maximal tori are conjugate in $\overline{G}$ and a quotient $N_{\overline{G}}(\overline{T})/\overline{T}$ is isomorphic to the Weyl group $W$ of $\overline{G}$. 

The natural question is to describe groups $\overline{G}$, in which $N_{\overline{G}}(\overline{T})$ splits over $\overline{T}$.
A similar question can be formulated for finite groups of Lie type. More precisely,
let $G$ be a finite group of Lie type, that is $O^{p'}(\overline{G}_{\sigma})\leqslant G\leqslant\overline{G}_{\sigma}$. Let $T=\overline{T}\cap G$ be a maximal torus in $G$ and $N(G,T)=N_{\overline{G}}(\overline{T})\cap G$ be the algebraic normalizer of $T$. Then the question is to describe groups $G$ and their maximal tori $T$ such that $N(G,T)$ splits over $T$.

These questions were stated by J.\,Tits in~\cite{Tits}. In the case of algebraic groups it was solved independently in~\cite{AdamsHe} and in~\cite{Galt1,Galt2,Galt3,Galt4}. In the case of finite groups the problem was
studied  for the groups of Lie types $A_n, B_n, C_n$ and $D_n$ in~\cite{Galt2,Galt3,Galt4}.

In this paper we consider finite groups $G$ of Lie type $E_6$ over a finite field $\F_q$ of characteristic $p$. Recall that in this case $G/Z(G)$ is isomorphic to the simple group $E_6(q)$ and either $|Z(G)|=1$ or $|Z(G)|=3$.  There are 25 conjugacy classes of maximal tori in $E_6(q)$ and we enumerate them as in~\cite{DerF}. Let 
$\Delta=\{r_1,r_2,r_3,r_4,r_5,r_6\}$ be a fundamental system of a root system $E_6$ and $r_{14}=r_2+r_4+r_5$, $r_{31}=r_2+r_3+2r_4+2r_5+r_6$, $r_{36}=r_1+2r_2+2r_3+3r_4+2r_5+r_6$. Denote by $w_i$ the elements of the Weyl group $W$ corresponding to reflections in the hyperplanes orthogonal to the roots $r_i$. There is a bijection between conjugacy classes of maximal tori and conjugacy classes of $W$. 
The main result of the paper is the following

\begin{theorem}\label{th1}
Let $G$ be a finite group of Lie type $E_6$ over ${\F}_q$ (adjoint or simply connected) with the Weyl group $W$. Let $T$ be a maximal torus of $G$ corresponding to an element $w$ of $W$ and $N$ its algebraic normalizer. Then $T$ does not have a complement in~$N$ if and only if one of the following claims holds:
\begin{itemize}
  \item[{\em (i)}] $q$ is odd and $w$ is conjugate to one of the following:
$1$, $w_1$, $w_1w_2$, $w_2w_3w_5$, $w_1w_3w_4$, $w_1w_4w_6w_{36}$, $w_1w_4w_6w_3$, $w_1w_4w_6w_3w_{36}$;
  \item[{\em (ii)}] $q\equiv3\pmod4$ and $w$ is conjugate to $w_3w_2w_4w_{14}$.
  \end{itemize}
\end{theorem}

In the end of the paper we illustrate this result in Table~\ref{table} with some useful information about the maximal tori. 
In fact, when $T$ has a complement in~$N$ we construct it explicitly.

J.\,Adams and X.\,He~\cite{AdamsHe} considered a related problem. Namely, it is natural to ask about the orders of lifts of $w\in W$ to $N_{\overline{G}}(\overline{T})$. They noticed that if $d$ is the order of $w$ then the minimal order of a lift of $w$ has order $d$ or $2d$, but it can be a subtle question which holds. Clearly, if $N_{\overline{G}}(\overline{T})$ splits over $\overline{T}$ then the minimal order is equal to $d$. In the case of Lie type $E_6$ it was proved that the normalizer does not split and they showed that the minimal order of lifts of $w$ is $d$ if $w$ belongs to so-called regular or elliptic conjugacy classes.
We prove that the minimal order of lifts of $w$ is always equal to the order $w$ in this case. Moreover, we construct these lifts in the corresponding algebraic normalizers.

\begin{theorem}\label{th2}
Let $G$ be a finite group of Lie type $E_6$ over ${\F}_q$ (adjoint or simply connected)  with the Weyl group $W$. Let $T$ be a maximal torus of $G$ corresponding to an element $w$ of $W$ and $N$ its algebraic normalizer.
Then there exists a lift for $w$ in $N$ with the same order.
\end{theorem}

This paper is organized as follows. In Section 2 we recall notation and basic facts about algebraic groups. In Section 3 we prove auxiliary results and explain how we use MAGMA in the proofs. Section 4 is devoted to the proof of the main results. It consists of two subsections. First, we consider the maximal tori that do not have a complement. In these cases we also prove the existence of the lifts of the corresponding elements of $W$. In the second subsection we construct the complements for the remaining tori. 

\section{Notation and preliminary results}

By $q$ we always denote some power of a prime $p$, $\overline{\F}_p$ is an algebraic closure of a finite field $\F_p$. The symmetric group on $n$ elements is denoted by $S_n$. Let $T$ be a normal subgroup of a group $N$. Then {\em $T$ has a complement in $N$} (or {\em $N$ splits over $T$}) if there exists a subgroup $H$ such that $N=T H$ and $T\cap H=1$.

By $\overline{G}$ we denote a simple connected linear algebraic group over $\overline{\F}_p$. 
A surjective endomorphism $\sigma$ of $\overline{G}$ is called a {\em Steinberg endomorphism}, if the set of $\sigma$-stable points $\overline{G}_\sigma$ is finite \cite[Definition~1.15.1]{GorLySol}. Any group $G$ satisfying $O^{p'}(\overline{G}_\sigma)\leqslant G\leqslant\overline{G}_\sigma$, is called a {\em finite group of Lie type}.
It is well known that $\overline{G}$ always has a $\sigma$-stable maximal torus, which is denoted by $\overline{T}$. All maximal tori are conjugate to $\overline{T}$ in $\overline{G}$.
If $\overline{T}$ is a $\sigma$-stable maximal torus of $\overline{G}$ then $T=\overline{T}\cap G$ is called a {\em maximal torus} of $G$. If $G_2\unlhd G_1\unlhd G$ then the image of $T\cap G_1$ in $G_1/G_2$ is called a {\em maximal torus} of $G_1/G_2$. A group $N_{\overline{G}}(\overline{T})\cap G$ is denoted by $N(G,T)$ or just $N$. Notice that $N(G,T)\leqslant N_G(T)$, but the equality is not true in general. For example, let $G=\SL_n(2)$  then the subgroup of diagonal matrices $T$ of $G$ is trivial, hence, $N_G(T)=G$. But $G=(\SL_n(\overline{\F}_2))_\sigma$, where $\sigma$ is a Frobenius map $(\sigma : (a_{i,j}) \mapsto (a_{i,j}^2))$. Then $T=\overline{T}_\sigma$, where $\overline{T}$ is the subgroup of diagonal matrices in  $\SL_n(\overline{\F}_2)$. Thus $N(G,T)$  is the group of monomial matrices in $G$. That is why for the group $N(G,T)$ we use the term {\em algebraic normalizer}.

By $\overline{N}$ and $W$ we denote the normalizer $N_{\overline{G}}(\overline{T})$ and the Weyl group $\overline{N}/\overline{T}$, and $\pi$ stands for the natural homomorphism from $\overline{N}$ onto $W$. Define the action of $\sigma$ on $W$ in the natural way. Elements $w_1, w_2\in W$ are called $\sigma$-conjugate if $w_1=(w^{-1})^{\sigma}w_2w$ for some $w\in W$.

\begin{prop}{\em\cite[Propositions 3.3.1, 3.3.3]{Car}}\label{torus}.
A torus $\overline{T}^g$ is $\sigma$-stable if and only if $g^{\sigma}g^{-1}\in\overline{N}$. The map $\overline{T}^g\mapsto\pi(g^{\sigma}g^{-1})$ determines a bijection between the $G$-classes of $\sigma$-stable maximal tori of $\overline{G}$ and the $\sigma$-conjugacy classes of $W$.
\end{prop}

As follows from Proposition~\ref{torus} the cyclic structure of a torus $(\overline{T}^g)_{\sigma}$ in $G$ and of the corresponding tori
in the sections of $G$ is determined only by a $\sigma$-conjugacy class of the element $\pi(g^{\sigma}g^{-1})$.

\begin{prop}{\em\cite[Lemma 1.2]{ButGre}}\label{prop2.5}.
Let $n=g^{\sigma}g^{-1}\in\overline{N}$. Then $(\overline{T}^g)_\sigma=(\overline{T}_{\sigma n})^g$, where $n$ acts on $\overline{T}$ by conjugation.
\end{prop}

\begin{prop}{\em\cite[Proposition 3.3.6]{Car}}\label{normalizer}.
Let $g^{\sigma}g^{-1}\in\overline{N}$ and $\pi(g^{\sigma}g^{-1})=w$. Then $$(N_{\overline{G}}({\overline{T}}^g))_{\sigma}/({\overline{T}}^g)_{\sigma}\simeq C_{W,\sigma}(w)=\{x\in W~|~(x^{-1})^{\sigma}wx=w\}.$$
\end{prop}

From now on, we suppose that $\overline{G}$ has Lie type $E_6$.
Then the Weyl group $W$ has order $2^{7}\cdot3^4\cdot5$
and is isomorphic to the group $\PSp_4(3):2$ (in the notation of~\cite{Atlas}).
Let $\Phi$ and $\Pi=\{r_1,r_2,r_3,r_4,r_5,r_6\}$ be a set of positive and fundamental roots of a root system $E_6$, respectively. The Dynkin diagram of $E_6$ has the form

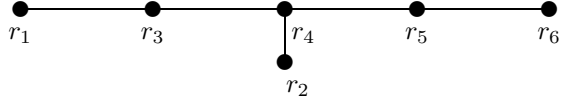
\begin{picture}(330,50)(-110,-30)
\put(0,0){\line(1,0){50}} \put(50,0){\line(1,0){50}}
\put(100,0){\line(1,0){50}} \put(150,0){\line(1,0){50}}
\put(0,0){\circle*{6}} \put(50,0){\circle*{6}}
\put(100,0){\circle*{6}} \put(150,0){\circle*{6}}
\put(200,0){\circle*{6}}
\put(100,0){\line(0,-1){20}}
\put(100,-20){\circle*{6}}
\put(0,-10){\makebox(0,0){$r_1$}}
\put(50,-10){\makebox(0,0){$r_3$}}
\put(107,-10){\makebox(0,0){$r_4$}}
\put(150,-10){\makebox(0,0){$r_5$}}
\put(200,-10){\makebox(0,0){$r_6$}}
\put(105,-30){\makebox(0,0){$r_2$}}
\end{picture}\\

Following~\cite{Car}, we write $x^y=yxy^{-1}$ and $[x,y]=y^xy^{-1}$. 
It is well known \cite{Car} that
\begin{center}
$\overline{T}=\langle h_r(\lambda)~|~r\in\Phi,\lambda\in \overline{\F}_p^*\rangle, \quad
\overline{N}=\langle \overline{T},n_r~|~r\in\Phi\rangle$,
\end{center}
where
$$
n_r(1)=n_r,\quad
n_r^2=h_r(-1),\quad
h_r(\lambda)=n_r(\lambda)n_r(-1).$$
For simplicity of notation, we write $h_r$ for $h_r(-1)$. If $r=r_i$, then $h_i$ stands for $h_{r_i}$ and $n_i$ stands for $n_{r_i}$.
Every element $H$ of $\overline{T}$ can be written in the form $H=\prod\limits_{i=1}^6h_{r_i}(\lambda_i)$.
Then the element $H$ is uniquely determined by $\lambda_1,\lambda_2,\lambda_3,\lambda_4,\lambda_5,\lambda_6$, and we
write $H = (\lambda_1,\lambda_2,\lambda_3,\lambda_4,\lambda_5,\lambda_6)$.

The group $\mathcal{T}=\langle n_r~|~ r\in\Pi\rangle$ is called {\it the Tits group}. Let $\mathcal{H}=\overline{T}\cap\mathcal{T}$. It is known
that $\mathcal{H}=\langle h_r~|~r\in\Pi\rangle$ and so $\mathcal{H}$ is an elementary abelian group such that $\mathcal{T}/\mathcal{H}\simeq W$. Observe that if $p=2$ then $h_r=1$ for every $r\in\Pi$, in particular, $\mathcal{H}=1$ and $\mathcal{T}\simeq W$. This implies the assertions of the both theorems, when $q$ is even (for details see Section 4.2). 

Let $\xi\in\overline{\F}_p$ such that $\xi^3=1$. According to Table $1.12.6$ in~\cite{GorLySol} the centre
$Z(E_6(\overline{\F}_p))$ of the simply connected group $E_6(\overline{\F}_p)$ is generated by the element
$z=h_{r_1}(\xi)h_{r_3}(\xi^2)h_{r_5}(\xi)h_{r_6}(\xi^2)$ of order 3.

According~\cite[Theorem 7.2.2]{Car} we have:
\begin{center}
$n_s n_r n_s^{-1}=h_{w_s(r)}(\eta_{s,r})n_{w_s(r)}=
n_{w_s(r)}(\eta_{s,r}),\quad \eta_{s,r}=\pm1,$
\end{center}
\begin{center}
$n_s h_r(\lambda)n_s^{-1}=h_{w_s(r)}(\lambda).$
\end{center}

We choose values of $\eta_{r,s}$ as follows. 
Let $r\in\Phi$ and $r=\sum\limits_{i=1}^6\alpha_i r_i$. The sum of the coefficients $\alpha_1+\alpha_2+\alpha_3+\alpha_4+\alpha_5+\alpha_6$ 
is called the height of $r$. Following to \cite{Vavilov}, we fix the following total ordering of positive roots: we write $r\prec s$ if either $h(r)<h(s)$ or 
$h(r)=h(s)$ and the first nonzero coordinate of $s-r$ is positive.
The table of positive roots with respect to this ordering can be found in \cite{Vavilov}.

Recall that a pair of positive roots $(r,s)$ is called {\it special} if $r+s\in\Phi$ and $r\prec s$. 
A pair $(r,s)$ is called {\it extraspecial} if it is special and for any special pair $(r_1, s_1)$ such that $r+s=r_1+s_1$ one has $r\preccurlyeq s$. Let $N_{r,s}$ be the structure constants of the corresponding simple Lie algebra. Then the values $N_{r,s}$ may be taken arbitrarily at the extraspecial pairs and then all other structure constants are uniquely determined \cite[Proposition~4.2.2]{Car}. In our case, we choose $\sgn(N_{r,s})=+$ for all extraspecial pairs $(r,s)$.
Values of the structure constants for all pairs can be found in \cite{Vavilov}.
The numbers $\eta_{r,s}$ are uniquely determined by the structure constants \cite[Proposition~6.4.3]{Car}.

\section{Preliminaries: calculations}

We use MAGMA to calculate products of elements in $\overline{N}$. 
All calculations can be performed using online Magma Calculator \cite{MC} as well.
At the moment it uses Magma V2.23-10. We use the following preparatory commands:\newline
$L:=LieAlgebra("E6", Rationals());$\newline
$R:=RootDatum(L);$\newline
$B:=ChevalleyBasis(L);$

The following command produces the list of extraspecial pairs and signs of the corresponding structure constants:\newline
{\small $x, y, h:= ChevalleyBasis(L); IsChevalleyBasis(L, R, x, y, h);$\newline
$[\langle1,3,1\rangle$, $\langle1,9,1\rangle$, $\langle1,13,1\rangle$, $\langle1,15,1\rangle$, $\langle1,19,1\rangle$, $\langle1,21,1\rangle$, $\langle1,24,1\rangle$, $\langle1,25,1\rangle$, $\langle1,28,1\rangle$, 
$\langle1,31,1\rangle$, $\langle2,4,1\rangle$, $\langle2,9,1\rangle$, $\langle2,10,1\rangle$, $\langle2,15,1\rangle$, $\langle2,16,1\rangle$, $\langle2,21,1\rangle$, $\langle2,35,1\rangle$, $\langle3,4,1\rangle$, 
$\langle3,10,1\rangle$, $\langle3,16,1\rangle$, $\langle3,26,1\rangle$, $\langle3,30,1\rangle$, $\langle3,33,1\rangle$, $\langle4,5,1\rangle$, $\langle4,11,1\rangle$, $\langle4,19,1\rangle$, $\langle4,25,1\rangle$, $\langle4,34,1\rangle$, $\langle5,6,1\rangle$, $\langle5,28,1\rangle]$.}

Here, for example, $\langle2,4,1\rangle$ means that the pair $(r_2,r_4)$ is extraspecial and $N_{r_2,r_4}=1$.
It is straightforward to verify the defined above ordering gives the same set of extraspecial pairs.
Thus calculations in MAGMA for $\overline{N}$ correspond to the ordering and structure constants defined in the previous section. The following commands construct elements $n_i$ and $h_i$. \newline
$G:=GroupOfLieType(L);$ \newline
$n:=[ elt\langle G~|~i\rangle : i\ in\ [1..36]];$ \newline
$h:=[TorusTerm(G, i,-1) : i\ in\ [1..36]];$

To obtain the list of matrices of the fundamental reflections one can use the following command:\newline
$w:=[Transpose(i)\ :\ i\ in\ ReflectionMatrices(R)];$

\begin{lem}\label{normalizer}
Let $g\in\overline{G}$ and $n=g^\sigma g^{-1}\in\overline{N}$. Suppose that $H\in \overline{T}$ and $u\in\mathcal{T}$. Then

(i) $Hu\in\overline{N}_{\sigma n}$ if and only if   $H=H^{\sigma n}[n,u];$

(ii) If $H\in \mathcal{H}$ then $Hu\in\overline{N}_{\sigma n}$ if and only if   $[n,Hu]=1$.
\end{lem}
\begin{proof} (i)
Since $\sigma$ acts trivially on $\mathcal{T}$, we have  $Hu=(Hu)^{\sigma n}=H^{\sigma n}u^{\sigma n}=
H^{\sigma n}u^{n}$ and $H=H^{\sigma n}[n,u]$.

(ii) Since $\sigma$ acts trivially on $\mathcal{H}$, we have $Hu=(Hu)^{\sigma n}=(Hu)^{n}.$
\end{proof}

\begin{lem}\label{conjugation}
 Let $n=\prod\limits_{j=1}^{k}n_{r_{i_j}}$ and $w=\prod\limits_{j=1}^{k}w_{r_{i_j}}$, where $i_j\in\{1..36\}$. Suppose that
$A=(a_{ij})_{6\times6}$ is the matrix of $w$ in the basis $r_1,r_2,...,r_6$ and $H=(\lambda_1,\lambda_2,\lambda_3,\lambda_4,\lambda_5,\lambda_6)$ is an element of $\overline{T}$. Then the following claims hold:
 
(i) $H^n=(\lambda_1',\lambda_2',\lambda_3',\lambda_4',\lambda_5',\lambda_6')$,
where $\lambda_i'=\lambda_1^{a_{i1}}\lambda_2^{a_{i2}}\lambda_3^{a_{i3}}\lambda_4^{a_{i4}}\lambda_5^{a_{i5}}\lambda_6^{a_{i6}}$ for $1\leqslant{i}\leqslant6$;

(ii) $(Hn)^m=(\lambda_1',\lambda_2',\lambda_3',\lambda_4',\lambda_5',\lambda_6')n^m$,
where $m$ is a positive integer,  $\lambda_i'=\lambda_1^{b_{i1}}\lambda_2^{b_{i2}}\lambda_3^{b_{i3}}\lambda_4^{b_{i4}}\lambda_5^{b_{i5}}\lambda_6^{b_{i6}}$ for $1\leqslant{i}\leqslant6$ and $b_{ij}$ are elements of the matrix $\sum\limits_{k=0}^{m-1}A^k$.
\end{lem}
\begin{proof} Since the matrix of the composition of two linear transformations is the product of their matrices  in the reverse order, it is sufficient to proof the lemma for $n=n_r$.
By formulas, $H^{n_r}=\prod\limits_{i=1}^6h_{r_i}(\lambda_i)^{n_r}=\prod\limits_{i=1}^6h_{w_r(r_i)}(\lambda_i)=\prod\limits_{i=1}^6(\prod\limits_{j=1}^6h_{r_j}(\lambda_i^{a_{ji}}))=
\prod\limits_{i=1}^6h_{r_i}(\prod\limits_{j=1}^6\lambda_j^{a_{ij}})$.

To prove $(ii)$ observe that $(Hn)^m=H^{n^0}H^{n^1}H^{n^2}..H^{n^{m-1}}n^m$. By $(i)$, we know that
the $i$-th row of $A^j$ corresponds to the exponents of $\lambda_1$, $\lambda_2$, $\ldots$ , $\lambda_6$ in $i$-th coordinate of $H^{n^j}$ for every $i,j\in\{1,\ldots,6\}$. To compute the product, we need to sum the exponents up for each coordinate. 
The lemma is proved.
\end{proof}
Since we often use Lemma~\ref{conjugation}, we illustrate its applying with the following example.

\noindent{\bf Example.} Let $w=w_1w_3$ and $n=n_1n_3$. Then it is easy to see that
$w(r_1)=r_3$, $w(r_2)=r_2$, $w(r_3)=-r_1-r_3$, $w(r_4)=r_1+r_3+r_4$, $w(r_5)=r_5$ and $w(r_6)=r_6$.
Therefore, in this case 
$$A=\left(\begin{array}{cccccc} 0 & 0 & -1 & 1 & 0 & 0 \\ 0 & 1 & 0 & 0 & 0 & 0 \\
1 & 0 & -1 & 1 & 0 & 0 \\  0 & 0 & 0 & 1 & 0 & 0 \\  0 & 0 & 0 & 0 & 1 & 0 \\  0 & 0 & 0 & 0 & 0 & 1
\end{array} \right) .$$
Let $H=(\lambda_1, \lambda_2, \lambda_3, \lambda_4,\lambda_5,\lambda_6)\in\overline{T}$. Then by Lemma~\ref{conjugation},
we can use the rows of $A$ to compute $H^n$, namely $H^{n_1n_3}=(\lambda_3^{-1}\lambda_4,\lambda_2,\lambda_1\lambda_3^{-1}\lambda_4,\lambda_4,\lambda_5,\lambda_6)$.
Now let $B=A^0+A+A^2$. Then
$$B=\left(\begin{array}{cccccc} 0 & 0 & 0 & 1 & 0 & 0 \\ 0 & 3 & 0 & 0 & 0 & 0\\
0 & 0 & 0 & 2 & 0 & 0 \\ 0 & 0 & 0 & 3 & 0 & 0\\  0 & 0 & 0 & 0 & 3 & 0 \\ 0 & 0 & 0 & 0 & 0 & 3
\end{array} \right) .$$
This matrix helps us to compute $(Hn)^3$. It is easy to see that $n^3=1$, so $$(Hn)^3=(\lambda_4,\lambda_2^3,\lambda_4^2,\lambda_4^3,\lambda_5^3,\lambda_6^3)n^3=(\lambda_4,\lambda_2^3,\lambda_4^2,\lambda_4^3,\lambda_5^3,\lambda_6^3).$$

The following lemma is clear.
\begin{lem}\label{commutator}
Let $N_1=H_1u_1, N_2=H_2u_2$, where $H_1, H_2\in T$ and $u_1,u_2\in N$. Then
$$N_1N_2=N_2N_1 \text{ if and only if } H_1^{-1}H_1^{u_2}\cdot u_2u_1u_2^{-1}u_1^{-1}=H_2^{-1}H_2^{u_1}.$$ 
In particular, if $[u_1,u_2]=1$, then
$N_1N_2=N_2N_1 \text{ if and only if } H_1^{-1}H_1^{u_2}=H_2^{-1}H_2^{u_1}.$
\end{lem}

\section{Proof of the main results}
The proof of Theorem~\ref{th1} is divided into two subsections. First, we consider maximal tori that do not have a complement in their algebraic normalizer. In these cases we show that the corresponding element of $W$ has preimage in $N$ of the same order. The second subsection is devoted to the construction of the complements for the remaining maximal tori. In these cases the lifts obviously exist in the complements. We use the numeration of the maximal tori as in~\cite[Table~I]{DerF} and Table~\ref{table}. All calculations in $W$ can be verified in MAGMA~\cite{MAGMA} or GAP~\cite{GAP}.

\subsection{Non-complement cases} 
Our strategy is similar in all cases. Throughout this subsection we suppose that $q$ is odd and $T$ is a maximal torus corresponding to the conjugacy class of $w$ in $W$, where $w$ is one from the assertion of Theorem~\ref{th1}. We suppose that there exists a complement $K$ in $N$, in particular $K\simeq C_W(w)$. 
We arrive at a contradiction in each case in the universal group $E_6(q)$. Then we explain that the same contradiction can be obtained in the adjoint group $E_6(q)$.
Throughout this subsection 
$H=(\lambda_1,\lambda_2,\lambda_3,\lambda_4,\lambda_5,\lambda_6)$ is an arbitrary element of $T$.

\textbf{Torus 1.} In this case $w=1$ and $C_W(w)=W$. It was proved \cite[Corollary]{Galt1} that the corresponding maximal torus in algebraic group does not have a complement in $\overline{N}$. Therefore, $T$ does not have a complement in $N$ as well. 

\textbf{Tori 2, 3.} In these cases $w=w_1$ or $w=w_1w_2$, respectively. Observe that both centralizers 
$C_W(w_1)$ and $C_W(w_1w_2)$ contain a subgroup $\langle w_1,w_2,w_5,w_{29}\rangle$ and we deduce a contradiction from it.
Let $N_1,N_2,N_3,N_4$ be preimages of $w_1,w_2,w_5,w_{29}$ in $K$, respectively. Then
$$N_1=H_1n_1,N_2=H_2n_2,N_3=H_3n_5,N_4=H_4n_{29},$$ where
$$H_1=(\mu_1,\mu_2,\mu_3,\mu_4,\mu_5,\mu_6),
H_2=(\alpha_1,\alpha_2,\alpha_3,\alpha_4,\alpha_5,\alpha_6),
H_3=(\beta_1,\beta_2,\beta_3,\beta_4,\beta_5,\beta_6),$$
$H_4=(\delta_1,\delta_2,\delta_3,\delta_4,\delta_5,\delta_6)$ are elements of $T$.
Since $K\simeq C_W(w)$, we have $N_1^2=1$ and $N_1N_i=N_iN_1$ for $i=2,3,4.$
By Lemma~\ref{conjugation}, we get
$$N_1^2=(H_1n_1)^2=H_1H_1^{n_1}h_1=(-\mu_3, \mu_2^2,\mu_3^2,\mu_4^2,\mu_5^2,\mu_6^2)=1,$$
so $1=\mu_2^2=\mu_4^2=\mu_5^2=\mu_6^2=-\mu_3$.
If $j\in\{2,5,29\}$ then using MAGMA~\cite{MAGMA} we see that $[n_1,n_j]=1$. It follows from Lemma~\ref{commutator} that $H_1^{-1}H_1^{n_j}=H_j^{-1}H_j^{n_1}$.
By Lemma~\ref{conjugation} we conclude
$$j=2 \Rightarrow (1,\mu_2^{-2}\mu_4,1,1,1,1)=(\alpha_1^{-2}\alpha_3,1,1,1,1,1),
\text{ so } \mu_4=\mu_2^2=1.$$
$$j=5 \Rightarrow (1,1,1,1,\mu_4\mu_5^{-2}\mu_6,1)=(\beta_1^{-2}\beta_3,1,1,1,1,1),
\text{ so } \mu_6=\mu_5^2\mu_4^{-1}=1.$$
$$j=29 \Rightarrow (\mu_3^{-1}\mu_6,\mu_3^{-1}\mu_6,\mu_3^{-2}\mu_6^2,\mu_3^{-2}\mu_6^2,\mu_3^{-1}\mu_6,1)= (\delta_1^{-2}\delta_3,1,1,1,1,1),$$ so $\mu_3=\mu_6=1.$
This contradicts $\mu_3=-1$.

In the adjoint group $E_6(q)$, an element $H=(\lambda_1,\lambda_2,\lambda_3,\lambda_4,\lambda_5,\lambda_6)$ is the identity if and only if $\lambda_1=\lambda_5=\xi,\lambda_3=\lambda_6=\xi^2,\lambda_2=\lambda_4=1,$ where
$\xi^3=1$. In particular, $\lambda_i^3=1$. Note that $(-1)^3=-1$. Hence, we can consider the elements
$\widetilde{\lambda}_i=\lambda_i^3$ for $\lambda_i\in\{\mu_i,\alpha_i,\beta_i,\delta_i\}$ with $1\leqslant i\leqslant6$. Then we obtain the same equalities and the same contradiction.

Now we provide the lifts for $w_1$ and $w_1w_2$ to prove Theorem~\ref{th2} in these cases.  
Let $\zeta$ be an element of $\overline{\F}_p$ such that $\zeta^{q+1}=-1$. Put $H_1=(\zeta,1,-1,1,1,1)$. We claim that $H_1n_1\in N$  and $(H_1n_1)^2=1$.
By Lemma~\ref{conjugation}, we have 
$H^{n_1}=(\lambda_1^{-1}\lambda_3,\lambda_2,\lambda_3,\lambda_4,\lambda_5,\lambda_6)$.
Hence, $H_1^{\sigma n_1}=(-\zeta^{-q},1,-1,1,1,1)=H_1,$ and so $H_1\in T$.
Using MAGMA~\cite{MAGMA}, 
$$(Hn_1)^2=(-\lambda_3,\lambda_2^2,\lambda_3^2,\lambda_4^2,\lambda_5^2,\lambda_6^2).$$

Thus $(H_1n_1)^2=1$ and $H_1n_1$ is a required lift for $w_1$. 

Similarly, put
$H_2=(\zeta,\zeta,-1,-1,1,1)$. By Lemma~\ref{conjugation}, we have 
$$H^{n_1n_2}=(\lambda_1^{-1}\lambda_3,\lambda_2^{-1}\lambda_4,\lambda_3,\lambda_4,\lambda_5,\lambda_6).$$
Hence, $H_1^{\sigma n_1n_2}=(-\zeta^{-q},-\zeta^{-q},-1,-1,1,1)=H_2,$ and so $H_2\in T$. Using MAGMA~\cite{MAGMA}, 
$$(Hn_1n_2)^2=(-\lambda_3,-\lambda_4,\lambda_3^2,\lambda_4^2,\lambda_5^2,\lambda_6^2).$$

Thus $(H_1n_1n_2)^2=1$ and $H_1n_1n_2$ is a required lift for $w_1w_2$.

\textbf{Torus 5.} In this case $w=w_2w_3w_5$ and $$C_W(w)=\langle w\rangle\times \langle w_{24}\rangle\times\langle x,y,z\rangle\simeq
\mathbb{Z}_2\times\mathbb{Z}_2\times S_4,$$ where
$$x=w_{17}w_{18},y=w_{20}w_{21},z=w_{16}w_{25}, \text{ and } x^2=y^2=z^2=(yz)^2=(xy)^3=(xz)^3=1.$$
Using MAGMA, we have $[n_{24},n_2n_3n_5]=[n_{17}n_{18},n_2n_3n_5]= [h_4h_6n_{20}n_{21},n_2n_3n_5]=[n_{16}n_{25},n_2n_3n_5]=1.$
Hence, Lemma~\ref{normalizer} (ii) yields $w_{24},x,y,z$ are images of  $n_{24},n_{17}n_{18}, h_4h_6n_{20}n_{21}, n_{16}n_{25}$, respectively.
Let $N_1,N_2,N_3,N_4$ be preimages of $w_2w_3w_5,w_{24},w_{20}w_{21},w_{16}w_{25}$
in $K$. Then
$$N_1=H_1n_2n_3n_5,N_2=H_2n_{24},N_3=H_3h_4h_6n_{20}n_{21},N_4=H_4n_{16}n_{25},$$
where
$$H_1=(\mu_1,\mu_2,\mu_3,\mu_4,\mu_5,\mu_6),
H_2=(\alpha_1,\alpha_2,\alpha_3,\alpha_4,\alpha_5,\alpha_6),
H_3=(\gamma_1,\gamma_2,\gamma_3,\gamma_4,\gamma_5,\gamma_6),$$
$H_4=(\delta_1,\delta_2,\delta_3,\delta_4,\delta_5,\delta_6).$
Since $K\simeq C_W(w)$, we get $N_2^2=1$. By Lemma~\ref{conjugation}, it follows that
$$N_2^2=H_2H_2^{n_{24}}h_2h_3h_5=(\alpha_1^2,-\alpha_1\alpha_2^2\alpha_4^{-1}\alpha_6,-\alpha_1\alpha_3^2\alpha_4^{-1}\alpha_6,
\alpha_1^2\alpha_6^2,-\alpha_1\alpha_4^{-1}\alpha_5^2\alpha_6,\alpha_6^2)=1,$$
so $\alpha_1^2=\alpha_6^2=1, \alpha_2^2=\alpha_3^2=\alpha_5^2=-\alpha_4\alpha_1^{-1}\alpha_6^{-1}$.

By Lemma~\ref{commutator}, $N_2N_1=N_1N_2$ is equivalent to $H_2^{-1}H_2^{n_2n_3n_5}=H_1^{-1}H_1^{n_{24}}$.
Then according to Lemma~\ref{conjugation} we have
$$(1,\alpha_2^{-2}\alpha_4,\alpha_1\alpha_3^{-2}\alpha_4,1,\alpha_4\alpha_5^{-2}\alpha_6,1)
=(1,\mu_1\mu_4^{-1}\mu_6,\mu_1\mu_4^{-1}\mu_6,
\mu_1^2\mu_4^{-2}\mu_6^2,\mu_1\mu_4^{-1}\mu_6,1).$$
Then $\mu_1\mu_4^{-1}\mu_6=\alpha_2^{-2}\alpha_4=\alpha_1\alpha_3^{-2}\alpha_4=\alpha_4\alpha_5^{-2}\alpha_6.$
Since $\alpha_2^2=\alpha_3^2=\alpha_5^2$, we obtain
$\alpha_1=1$ and $\alpha_6=1$. From $\alpha_2^2=-\alpha_4\alpha_1^{-1}\alpha_6^{-1}$, we get $\alpha_2^2=-\alpha_4.$

Calculations in MAGMA show $[n_{24},h_4h_6n_{20}n_{21}]=1.$ Hence,
$N_2N_3=N_3N_2$ implies that $H_2^{-1}H_2^{n_{20}n_{21}}=H_3^{-1}H_3^{n_{24}}$.
Then by Lemma~\ref{conjugation}
$$(1,\alpha_2^{-1}\alpha_3\alpha_6^{-1},\alpha_1\alpha_2\alpha_3^{-1}\alpha_6^{-1},
\alpha_1\alpha_6^{-2},\alpha_1\alpha_6^{-2},\alpha_1\alpha_6^{-2})=(1,\gamma_1\gamma_4^{-1}\gamma_6,\gamma_1\gamma_4^{-1}\gamma_6,\gamma_1^2\gamma_4^{-2}\gamma_6^2,\gamma_1\gamma_4^{-1}\gamma_6,1).$$

Since $\alpha_1=\alpha_6=1$, it follows that
$\alpha_2\alpha_3^{-1}=\gamma_1\gamma_4^{-1}\gamma_6$ and $1=\gamma_1\gamma_4^{-1}\gamma_6$.
Therefore $\alpha_2=\alpha_3.$
Using MAGMA we get $[n_{24},n_{16}n_{25}]=1.$ Hence,
$N_2N_4=N_4N_2$ implies $H_2^{-1}H_2^{n_{16}n_{25}}=H_4^{-1}H_4^{n_{24}}$.
Then by Lemma~\ref{conjugation}
$$(1,\alpha_1\alpha_2^{-1}\alpha_3^{-1}\alpha_4\alpha_6^{-1},\alpha_1\alpha_2^{-1}\alpha_3^{-1}\alpha_4\alpha_6^{-1},\alpha_1\alpha_6^{-2},\alpha_1\alpha_6^{-2},\alpha_1\alpha_6^{-2})=(1,\delta_1\delta_4^{-1}\delta_6,\delta_1\delta_4^{-1}\delta_6,\delta_1^2\delta_4^{-2}\delta_6^2,\delta_1\delta_4^{-1}\delta_6,1).
$$
Since $\alpha_1=\alpha_6=1$, it follows that
$\alpha_2^{-1}\alpha_3^{-1}\alpha_4=\delta_1\delta_4^{-1}\delta_6$ and $1=\delta_1\delta_4^{-1}\delta_6$.
Therefore, $\alpha_4=\alpha_2\alpha_3=\alpha_2^2;$ a contradiction with $\alpha_4=-\alpha_2^2$.

Now we provide the lift for $w$.
Let $\zeta$ be an element of $\overline{\F}_p$ such that $\zeta^{q+1}=-1$. Put $H_1=(1,\zeta,\zeta,-1,\zeta,1)$. We claim that $H_1n_2n_3n_5\in N$  and $(H_1n_2n_3n_5)^2=1$.
By Lemma~\ref{conjugation}, we have 
$H^{n_2n_3n_5}=(\lambda_1,\lambda_2^{-1}\lambda_4,\lambda_1\lambda_3^{-1}\lambda_4,
\lambda_4,\lambda_4\lambda_5^{-1}\lambda_6,\lambda_6)$.
Hence, $H_1^{\sigma n_2n_3n_5}=(1,-\zeta^{-q},-\zeta^{-q}-1,-\zeta^{-q},1)=H_1$
and so $H_1\in T$. Using MAGMA we get 
$$(Hn_2n_3n_5)^2=(\lambda_1^2, -\lambda_4, -\lambda_1\lambda_4, \lambda_4^2, -\lambda_4\lambda_6,\lambda_6^2)$$

Thus $(H_1n_2n_3n_5)^2=1$ and $H_1n_2n_3n_5$ is a required lift for $w_2w_3w_5$.

\textbf{Torus 7.} In this case $w=w_1w_3w_4$ and $$C_W(w)=~\langle w\rangle\times\langle w_6, w_{19}w_{26}
\rangle\simeq\mathbb{Z}_4 \times D_8.$$
Put $n=n_1n_3n_4$.
Let $N_2, N_3$ be preimages of $w_6$ and $w_{19}w_{26}$ in $K$, respectively.
Then $$N_2=H_2n_6, N_3=H_3n_{19}n_{26},$$
where
$H_2=(\mu_1,\mu_2,\mu_3,\mu_4,\mu_5,\mu_6)$ and $H_3=(\beta_1,\beta_2,\beta_3,\beta_4,\beta_5,\beta_6)$. Since $[n_1n_3n_4, n_6]=[n_1n_3n_4,n_{16}n_{26}]=1$,
Lemma~\ref{normalizer} (ii) yields $H_2, H_3\in T$.
Since $w_6^2=1$, we obtain $N_2^2=1$. By Lemma~\ref{conjugation}, $$(Hn_6)^2=(\lambda_1^2,\lambda_2^2,\lambda_3^2,\lambda_4^2,\lambda_5^2,-\lambda_5).$$
Therefore, we obtain $\mu_1^2=\mu_2^2=\mu_3^2=\mu_4^2=1$, $\mu_5=-1$.
In particular, $\mu_i^q=\mu_i^{-1}=\mu_i$ for $1\leqslant i\leqslant5$. 
By Lemma~\ref{conjugation},
$$H^{n}=(\lambda_2\lambda_4^{-1}\lambda_5, \lambda_2, \lambda_1\lambda_2\lambda_4^{-1}\lambda_5, \lambda_2\lambda_3\lambda_4^{-1}\lambda_5, \lambda_5, \lambda_6).$$
Since $H_2^{\sigma n}=H_2$,
we have $(\mu_2\mu_4\mu_5,\mu_2,\mu_1\mu_2\mu_4\mu_5,\mu_2\mu_3\mu_4\mu_5,\mu_5,\mu_6^q)=H_2$.
Using $\mu_5=-1$, we conclude that $\mu_2\mu_4=-\mu_1$, $\mu_1\mu_2\mu_4=-\mu_3$, $\mu_2\mu_3=-1$.
From the last equality we get $\mu_2=-\mu_3$. Therefore  $\mu_1=\mu_4$, $\mu_2=-1$ and $\mu_3=1$. Thus $H_2=(\mu_1,-1,1,\mu_1,-1,\mu_6)$.

Since $(w_{19}w_{26})^2=1$, we have $N_3^2=1$.
Applying MAGMA, we get $(n_{19}n_{26})^2=h_1h_4$ and
$$(Hn_{19}n_{26})^2=(-\lambda_1\lambda_3\lambda_4^{-1}\lambda_6,\lambda_2\lambda_5^{-1}\lambda_6^2,\lambda_2^{-1}\lambda_3^2\lambda_5^{-1}\lambda_6^2,-\lambda_1^{-1}\lambda_2^{-1}\lambda_3\lambda_4\lambda_5^{-1}\lambda_6^3, \lambda_2^{-1}\lambda_5\lambda_6^2,\lambda_6^2).$$
So $\beta_4=-\beta_1\beta_3\beta_6$, $\beta_2\beta_6^2=\beta_5$ and $\beta_6^2=1$. Therefore $\beta_2=\beta_5$.

Now $(w_6w_{19}w_{26})^4=1$ and so $(N_2N_3)^4=1$. Observe that $N_2N_3=H_2n_6H_3n_{19}n_{26}=H_2H_3^{n_6}n_6n_{19}n_{26}$.
By Lemma~\ref{conjugation}, 
$H^{n_6}=(\lambda_1,\lambda_2,\lambda_3,\lambda_4,\lambda_5,\lambda_5\lambda_6^{-1}).$
Applying to $H_2H_3^{n_6}$, we obtain  $$H_2H_3^{n_6}=(\mu_1,-1,1,\mu_1,-1,\mu_6)(\beta_1,\beta_2,\beta_3,\beta_4,\beta_2,\beta_2\beta_6^{-1})=(\mu_1\beta_1,-\beta_2,\beta_3,\mu_1\beta_4,-\beta_2,\mu_6\beta_2\beta_6^{-1}).$$ 
Since $(n_6n_{19}n_{26})^4=h_1h_4$, By Lemma~\ref{conjugation} we have
$$(Hn_6n_{19}n_{26})^4=(-\lambda_1^2\lambda_2^{-1}\lambda_3^2\lambda_4^{-2}\lambda_5,\ast).$$
So $1=(N_2N_3)^4=(-\mu_1^2\beta_1^2(-\beta_2^{-1})\beta_3^2\mu_1^{-2}\beta_4^{-2}(-\beta_2),\ast)$. Therefore, $-\beta_1^2\beta_3^2\beta_4^{-2}=1$, which is equivalent to $\beta_4^2=-\beta_1^2\beta_3^2$. As noted above, $\beta_4=-\beta_1\beta_3\beta_6$. Thus $\beta_4^2=\beta_1^2\beta_3^2\beta_6^2=\beta_1^2\beta_3^2$;
a contradiction.

Now we provide the lift for $w$.
Let $\zeta$ be an element of $\overline{\F}_p$ such that $\zeta^{2(q+1)}=-1$. Put $H_1=(-\zeta^{-q},-1,\zeta^{-q^2-q},\zeta,1,1)$. We claim that $H_1n\in N$  and $|H_1n|=4$.
Since $\zeta^{q^3+q^2+q+1}=\zeta^{(2q+2)((q^2+1)/2)}=-1$ and $-\zeta^{-q^3-q^2-q}=\zeta$,
we have
$$H_1^{\sigma n}=(-\zeta^{-q}, -1, \zeta^{-q^2-q}, -\zeta^{-q^3-q^2-q}, 1,1)=H_1.$$ 
Thus $H_1\in T$.
By Lemma~\ref{conjugation},
$$(Hn)^4=(-\lambda_2\lambda_5, \lambda_2^4, (\lambda_2\lambda_5)^2, -(\lambda_2\lambda_5)^3, \lambda_5^4,\lambda_6^4).$$

So $(H_1n)^4=1$ and $H_1n$ is a required lift for $w$.

\textbf{Torus 8.}
In this case $w=w_1w_4w_6w_{36}$ and $C_W(w)\geqslant\langle w_1,w_4,w_6,w_{36}\rangle \simeq\mathbb{Z}_2\times\mathbb{Z}_2\times\mathbb{Z}_2\times\mathbb{Z}_2$.

Put $n=n_1n_4n_6n_3$.  Using MAGMA, we see that $[n,n_1]=[n,n_4]=[n,n_6]=[n,n_{36}]=1$
and hence $n_1,n_4,n_6,n_{36}\in N$.
Let $N_1,N_2,N_3,N_4$ be preimages of $w_1,w_4,w_6,w_{36}$
in $K$, respectively. Then
$$N_1=H_1n_1,N_2=H_2n_4,N_3=H_3n_6,N_4=H_4n_{36},$$ where
$H_1=(\mu_1,\mu_2,\mu_3,\mu_4,\mu_5,\mu_6),$
$H_2=(\alpha_1,\alpha_2,\alpha_3,\alpha_4,\alpha_5,\alpha_6),$
$H_3=(\beta_1,\beta_2,\beta_3,\beta_4,\beta_5,\beta_6),$
$H_4=(\gamma_1,\gamma_2,\gamma_3,\gamma_4,\gamma_5,\gamma_6).$
Since $K\simeq C_W(w)$, we have $N_3^2=1$ and $N_3N_i=N_iN_3$, where $i=1,2,4.$
By Lemma~\ref{conjugation}, 
$$N_3^2=H_3H_3^{n_6}h_6=(\beta_1^2, \beta_2^2,\beta_3^2,\beta_4^2,\beta_5^2,-\beta_5)=1.$$
Therefore $1=\beta_1^2=\beta_2^2=\beta_3^2=\beta_4^2=-\beta_5$.
Using MAGMA we see that $[n_6,n_j]=1$ for each $j\in\{1,4,36\}$, whence 
$H_3^{-1}H_3^{n_j}=H_j^{-1}H_j^{n_6}$.
By Lemma~\ref{conjugation},
$$j=1 \Rightarrow (\beta_1^{-2}\beta_3,1,1,1,1,1)=(1,1,1,1,1,\mu_6^{-2}\mu_5),
\text{ whence } \beta_3=\beta_1^2=1.$$
$$j=4 \Rightarrow (1,1,1,\beta_2\beta_3\beta_4^{-2}\beta_5,1,1)=(1,1,1,1,1,\alpha_6^{-2}\alpha_5),
\text{ whence } \beta_4^2=\beta_2\beta_3\beta_5.$$
$$j=36 \Rightarrow (\beta_2^{-1},\beta_2^{-2},\beta_2^{-2},\beta_2^{-3},\beta_2^{-2},\beta_2^{-1})= (1,1,1,1,1,\gamma_6^{-2}\gamma_5), \text{ whence } \beta_2=1.$$
We derive a contradiction with $1=\beta_4^2=\beta_2\beta_3\beta_5=-1$.

Calculations in MAGMA show that $(n_1n_4n_6n_{36})^4=1$. Thus $n_1n_4n_6n_{36}$ is a required lift for $w$.

\textbf{Torus 11.} In this case $w=w_1w_4w_6w_3$ and 
$C_W(w)=\langle w, w_6,w_{36}\rangle\simeq\mathbb{Z}_4\times\mathbb{Z}_2\times\mathbb{Z}_2$.

Put $n=n_1n_4n_6n_3$.
Let $N_1,N_2,N_3$ be preimages of $w_1w_4w_6w_3,w_6,w_{36}$
in $K$, respectively. Then
$$N_1=H_1n,N_2=H_2n_6,N_3=H_3n_{36},$$ where
$H_1=(\mu_1,\mu_2,\mu_3,\mu_4,\mu_5,\mu_6),$
$H_2=(\beta_1,\beta_2,\beta_3,\beta_4,\beta_5,\beta_6),$
$H_3=(\alpha_1,\alpha_2,\alpha_3,\alpha_4,\alpha_5,\alpha_6).$
Using MAGMA we see that $[n,n_6]=[n,n_{36}]=1$ and hence $H_1, H_2, H_3\in~T$.
Since $K\simeq C_W(w)$, we have $N_2^2=1$ and $[N_3,N_1]=[N_3,N_2]=1$.
The calculations from the previous case (Torus 8) show that
$$N_2^2=1 \text{ implies } 1=\beta_1^2=\beta_2^2=\beta_3^2=\beta_4^2=-\beta_5,$$
and $N_3N_2=N_2N_3$ implies $\beta_2=1.$
By Lemma~\ref{commutator}, we obtain
$H_2^{-1}H_2^{n}=H_1^{-1}H_1^{n_{6}}$.
It follows from Lemma~\ref{conjugation} that
$$(\beta_1^{-1}\beta_3^{-1}\beta_4,1,\beta_1\beta_3^{-2}\beta_4,
\beta_1\beta_2\beta_3^{-1}\beta_4^{-1}\beta_5,1,\beta_5\beta_6^{-2})=(1,1,1,1,1,\mu_6^{-2}\mu_5).$$
Hence, $\beta_1\beta_4=\beta_3^2=1$ and $\beta_4=\beta_1\beta_3$. Therefore,
$1=\beta_1\beta_4=\beta_1\cdot\beta_1\beta_3=\beta_3$ and hence $\beta_4=\beta_1$. Moreover,
$1=\beta_1\beta_2\beta_3^{-1}\beta_4^{-1}\beta_5=\beta_1\beta_4^{-1}\beta_5=\beta_5;$
a contradiction with $\beta_5=-1$.

Now we provide the lift for $w$.
Let $\zeta$ be an element of $\overline{\F}_p$ such that $\zeta^{q^3+q^2+q+1}=~-1$. Put $H_1=(\zeta,-1,\zeta^{q+1},-\zeta^{-q^2},1,1)$. We claim that $H_1n\in N$  and $|H_1n|=4$.
By Lemma~\ref{conjugation}, we get 
$H^{n}=(\lambda_3^{-1}\lambda_4,\lambda_2,\lambda_1\lambda_3^{-1}\lambda_4,
\lambda_1\lambda_2\lambda_3^{-1}\lambda_5,\lambda_5,\lambda_5\lambda_6^{-1})$.
Hence, $H_1^{\sigma n}=(-\zeta^{-q^2-q-1},-1,-\zeta^{-q^2-q},-\zeta^{-q},1,1)^q=(-\zeta^{-q^3-q^2-q},-1,-\zeta^{-q^3-q^2},-\zeta^{-q^2},1,1)=H_1.$
Thus $H_1\in T$. Using MAGMA, we see that
$$(Hn)^4=(-\lambda_2\lambda_5, \lambda_2^4, (\lambda_2\lambda_5)^2, -(\lambda_2\lambda_5)^3, \lambda_5^4,\lambda_5^2).$$
So $(H_1n)^4=1$ and $H_1n$ is a required lift for $w$.

\textbf{Torus 14.} In this case $w=w_3w_2w_4w_{14}$ and $q\equiv-1\pmod{4}$.
Observe that $w_6w_{15}w_{20}\in C_W(w)$. 

Put $n=n_3n_2n_4n_{14}$. Using MAGMA, we see that $[n,h_6n_6n_{15}n_{20}]=1$, and therefore $h_6n_6n_{15}n_{20}\in N$.
Let $N_1$ and $N_2$ be preimages of $w$, $w_6w_{15}w_{20}$ in $K$, respectively.
Then $N_1=H_1n$ and  $N_2=H_2h_6n_6n_{15}n_{20}$, where $H_1=(\mu_1,\mu_2,\mu_3,\mu_4,\mu_5\,\mu_6)$,  $H_2=(\alpha_1,\alpha_2,\alpha_3,\alpha_4,\alpha_5,\alpha_6)$ are elements of $T$.
Using MAGMA, we get $(h_6n_6n_{15}n_{20})^4=h_2h_3$. Now Lemma~\ref{conjugation} implies that
$$(HN_2)^4=(\lambda_1^4,-\lambda_1\lambda_2^2\lambda_3^2\lambda_5^{-2},-\lambda_1^3\lambda_2^2\lambda_3^2\lambda_5^{-2},\lambda_1^4\lambda_4^4\lambda_5^{-4},\lambda_1^4,\lambda_1^2).$$

It follows that $-\alpha_5^2=\alpha_1\alpha_2^2\alpha_3^2$, $\alpha_1^2=1$. Since $[w,d]=1$, we have $H_1^{-1}{H_1}^{N_2}=H_2^{-1}{H_2}^{N_1}$. By Lemma~\ref{conjugation},
$$H^{n}=(\lambda_1,\lambda_3\lambda_4^{-1}\lambda_5,\lambda_1\lambda_3\lambda_4^{-1}\lambda_6,\lambda_3^2\lambda_4^{-1}\lambda_6, \lambda_2^{-1}\lambda_3\lambda_6,\lambda_6),$$
$$ H^{n_6n_{15}n_{20}}=(\lambda_1,\lambda_3\lambda_6^{-1},\lambda_1\lambda_3\lambda_5^{-1},\lambda_1\lambda_2^{-1}\lambda_3\lambda_4\lambda_5^{-1}\lambda_6^{-1},\lambda_1\lambda_2^{-1}\lambda_3\lambda_6^{-1},\lambda_1\lambda_6^{-1}).$$
Whence 
\begin{multline*}
(1,\mu_2^{-1}\mu_3\mu_6^{-1},\mu_1\mu_5^{-1},\mu_1\mu_2^{-1}\mu_3\mu_5^{-1}\mu_6^{-1},\mu_1\mu_2^{-1}\mu_3\mu_5^{-1}\mu_6^{-1},\mu_1\mu_6^{-2})=\\
=(1,\alpha_2^{-1}\alpha_3\alpha_4^{-1}\alpha_5,\alpha_1\alpha_4^{-1}\alpha_6,\alpha_3^2\alpha_4^{-2}\alpha_6,\alpha_2^{-1}\alpha_3\alpha_5^{-1}\alpha_6,1).
\end{multline*}
On the left-hand side, we see that the product of the second and third coordinates equals the fourth one, hence
$(\alpha_2^{-1}\alpha_3\alpha_4^{-1}\alpha_5)(\alpha_1\alpha_4^{-1}\alpha_6)=\alpha_3^2\alpha_4^{-2}\alpha_6$.
So we infer that $\alpha_1\alpha_5=\alpha_2\alpha_3$.
Since $-\alpha_5^2=\alpha_1\alpha_2^2\alpha_3^2$, we obtain $\alpha_1=-1$ and hence $\alpha_5=-\alpha_2\alpha_3$.
Moreover, values of the fourth and fifth coordinates coincide, so $\alpha_3^2\alpha_4^{-2}\alpha_6=\alpha_2^{-1}\alpha_3\alpha_5^{-1}\alpha_6$. Therefore, $\alpha_2\alpha_3\alpha_5=\alpha_4^2$ and hence $\alpha_5^2=-\alpha_4^2$.

Since $H_2$ belongs to the torus, we have $H_2^{\sigma{n}}=H_2$. Therefore $$(\alpha_1^q,(\alpha_3\alpha_4^{-1}\alpha_5)^q,(\alpha_1\alpha_3\alpha_4^{-1}\alpha_6)^q,(\alpha_3^2\alpha_4^{-1}\alpha_6)^q,(\alpha_2^{-1}\alpha_3\alpha_6)^q,\alpha_6^q)=(\alpha_1,\alpha_2,\alpha_3,\alpha_4,\alpha_5,\alpha_6).$$
Whence $\alpha_2=(\alpha_3\alpha_4^{-1}\alpha_5)^q$, $\alpha_3=(\alpha_1\alpha_3\alpha_4^{-1}\alpha_6)^q$ and $\alpha_4=(\alpha_3^2\alpha_4^{-1}\alpha_6)^q$. After squaring up the both sides of the equation for $\alpha_2$ and using $\alpha_4^2=-\alpha_5^2$, we have $-\alpha_3^{2q}=\alpha_2^2$. Since $\alpha_1=-1$,
it is true that $\alpha_4\alpha_3^{-1}=(\alpha_3^2\alpha_4^{-1}\alpha_6)^q(-\alpha_3\alpha_4^{-1}\alpha_6)^{-q}=-\alpha_3^q$ and hence $\alpha_4=-\alpha_3^{q+1}$.
Therefore, $\alpha_3=(\alpha_1\alpha_3\alpha_4^{-1}\alpha_6)^q=-\alpha_3^q\alpha_4^{-q}\alpha_6^q =-\alpha_3^q(-\alpha_3^{-q^2-q})\alpha_6$ and hence $\alpha_6=\alpha_3^{q^2+1}$. On the other hand, we have $\alpha_5=(\alpha_2^{-1}\alpha_3\alpha_6)^q$. Since $\alpha_6^q=\alpha_6$ and $\alpha_5=-\alpha_2\alpha_3$,
we obtain $\alpha_6=-\alpha_2^{q+1}\alpha_3^{1-q}$. We know that $\alpha_2^2=-\alpha_3^{2q}$ and hence $\alpha_2^{q+1}=\alpha_3^{q^2+q}$. Thus $\alpha_6=-\alpha_3^{q^2+1}$; a contradiction with $\alpha_6=\alpha_3^{q^2+1}$.

Calculations in MAGMA show that $n^4=1$ and hence $n$ is the required lift for $w$ in this case.

\textbf{Torus 16.} In this case $w=w_1w_4w_6w_3w_{36}$ and 
$C_W(w)=\langle w\rangle\times\langle w_6,w_{27},w_{36}\rangle\simeq \mathbb{Z}_4\times S_4$.
Observe that $(w_6w_{27})^3=(w_{36}w_{27})^3=(w_6w_{36})^2=1$ and $w_1w_4w_6w_3\in C_W(w)$.
Put $n=n_1n_4n_6n_3n_{36}$. Using MAGMA we see that $[n,n_1n_4n_6n_3]=[n,n_{36}]=[n,n_6]=1.$
Let $N_1,N_2,N_3$ be preimages of $w_1w_4w_6w_3,w_{36},w_{6}$ in $K$, respectively. Then
$$N_1=H_1n_1n_4n_6n_3,N_2=H_2n_{36},N_3=H_3n_{6},$$ where 
$$H_1=(\mu_1,\mu_2,\mu_3,\mu_4,\mu_5,\mu_6),H_2=(\alpha_1,\alpha_2,\alpha_3,\alpha_4,\alpha_5,\alpha_6),H_3=(\beta_1,\beta_2,\beta_3,\beta_4,\beta_5,\beta_6)$$ are elements of $T$.
Since $H\simeq C_W(w)$, we have $N_3^2=1$ and $[N_3,N_1]=[N_3,N_2]=1$.
The calculations for the Torus 8 show that
$$N_3^2=1 \text{ implies } 1=\beta_1^2=\beta_2^2=\beta_3^2=\beta_4^2=-\beta_5,$$
and $N_3N_2=N_2N_3$ implies $\beta_2=1.$
By Lemma~\ref{commutator}, the equation $N_3N_1=N_1N_3$ is equivalent to $H_3^{-1}H_3^{n_1n_4n_6n_3}=H_1^{-1}H_1^{n_{6}}$. By Lemma~\ref{conjugation}, the latter equation yields
$$(\beta_1^{-1}\beta_3^{-1}\beta_4,1,\beta_1\beta_3^{-2}\beta_4,
\beta_1\beta_2\beta_3^{-1}\beta_4^{-1}\beta_5,1,\beta_5\beta_6^{-2})=(1,1,1,1,1,\mu_6^{-2}\mu_5).$$
Hence, $\beta_1\beta_4=\beta_3^2=1$ and $\beta_4=\beta_1\beta_3$. So
$1=\beta_1\beta_4=\beta_1\cdot\beta_1\beta_3=\beta_3$ and $\beta_4=\beta_1$. Finally,
$$1=\beta_1\beta_2\beta_3^{-1}\beta_4^{-1}\beta_5=\beta_1\beta_4^{-1}\beta_5=\beta_5;$$
a contradiction with $\beta_5=-1$.

Now we provide the lift for $w$. Let $\zeta$ be an element of $\overline{\F}_p$ such that $\zeta^{q^3+q^2+q+1}=-1$. 
Put $H_1=(\zeta,1,\zeta^{q+1},-\zeta^{-q^2},-1,\zeta^{q^2+1})$. We claim that 
$H_1n\in N$ and $|H_1n|=4$.
By Lemma~\ref{conjugation}
$$H^{n}=(\lambda_2^{-1}\lambda_3^{-1}\lambda_4,\lambda_2^{-1},
\lambda_1\lambda_2^{-2}\lambda_3^{-1}\lambda_4,\lambda_1\lambda_2^{-2}\lambda_3^{-1}\lambda_5,\lambda_2^{-2}\lambda_5,\lambda_2^{-1}\lambda_5\lambda_6^{-1}).$$
Therefore 
$$H_1^{\sigma{n}}=(-\zeta^{-q^2-q-1},1,-\zeta^{-q^2-q},-\zeta^{-q},-1,-\zeta^{-q^2-1})^q
=(-\zeta^{-q^3-q^2-q},1,-\zeta^{-q^3-q^2},-\zeta^{-q^2},1,-\zeta^{-q^3-q})=H_1.$$
So $H_1\in T$. By Lemma~\ref{conjugation}, 
$$(Hn_1n_4n_6n_3n_{36})^4=(-\lambda_2^{-1}\lambda_5, 1, \lambda_2^{-2}\lambda_5^2, -\lambda_2^{-3}\lambda_5^3, \lambda_2^{-4}\lambda_5^4, \lambda_2^{-2}\lambda_5^2).$$

Thus $(H_1n)^4=1$ and $H_1n$ is a required lift for $w$.
\vspace{0.5em}

\subsection{Complement cases}
Now we will deal with maximal tori of $E_6(q)$ that have a complement in their algebraic normalizer.
Throughout this subsection we suppose that $T$ is a maximal torus corresponding to the conjugacy class of $w$ in $W$.
We use $w$ as in Table~\ref{table}. If $w=w_{i_1}w_{i_2}...w_{i_k}$, where $w_{i_j}$ are fundamental reflections, 
then we put $n=n_{i_1}n_{i_2}...n_{i_k}$. Observe that $n$ is a lift to $N$ for $w$. 

As was mentioned in Section 3, if $q$ is even then $\mathcal{T}\simeq W$. 
So, the orders of $n$ and $w$ are equal. If $K$ is the isomorphic copy of $C_W(w)$ in $\mathcal{T}$
and $x\in K$ then $[x,n]=1$ and hence by Lemma~\ref{normalizer} we have $x\in N$. Therefore, $K$ is a complement for $T$
in $N$.

Throughout this subsection we assume that $q$ is odd and $H=(\lambda_1,\lambda_2,\lambda_3,\lambda_4,\lambda_5,\lambda_6)$ is an arbitrary element of $T$.

\textbf{Torus 4.} In this case $w=w_3w_1$ and $C_W(w)=\langle w,w_5,w_6,w_2,
w_{36},v\rangle\simeq
\mathbb{Z}_3\times((S_3\times S_3):\mathbb{Z}_2)$, where $v=w_1w_4w_{14}w_{29}$, $w_5^v=w_2$, $w_6^v=w_{36}$ and
$\langle w_5,w_6\rangle\simeq\langle w_2,w_{36}\rangle\simeq S_3$. 
Using MAGMA we see that $[n,n_1n_4n_{14}n_{29}]=[n,n_5]=[n,n_6]=[n,n_2]=[n,n_{36}]=1$, whence $n_1n_4n_{14}n_{29},n_5,n_6,n_2,n_{36}\in N$.
Put $$N_1=n_1n_3, N_2=h_{36}n_2, N_3=h_2n_{36}, N_4=n_1n_4n_{14}n_{29}, N_5=h_5h_6n_5, N_6=h_5n_6.$$
We claim that $K=\langle N_1,N_2,N_3,N_4,N_5,N_6\rangle$ is a complement for $T$.
By Lemma~\ref{conjugation},
$$H^{n}=(\lambda_1^{-1}\lambda_3,\lambda_2,
\lambda_1^{-1}\lambda_4,\lambda_4,\lambda_5,\lambda_6).$$ 
Then $h_{36}^{\sigma{n}}=(-1,1,1,-1,1,-1)^{\sigma{n}}=h_{36}^\sigma=h_{36}.$
Similarly, $h_2^{\sigma{n}}=h_2$, $h_5^{\sigma{n}}=h_5$, $h_6^{\sigma{n}}=h_6$.
Therefore $h_{36}, h_2, h_5$ and $h_6$ belong to $T$.
Using MAGMA we have
$$N_2^2=N_3^2=N_5^2=N_6^2=1,\quad (N_2N_3)^3=1, (N_5N_6)^3=1,$$
that is $\langle N_5,N_6\rangle\simeq\langle N_2,N_3\rangle\simeq S_3.$
Further, using MAGMA we obtain
$$N_1^3=1,\quad N_4^2=1,\quad N_1^{N_2}=N_1^{N_3}=N_1^{N_4}=N_1^{N_5}=N_1^{N_6}=N_1.$$
Hence, $K=\langle N_1\rangle\times\langle N_2,N_3,N_4,N_5,N_6\rangle$.
Finally, calculations in MAGMA show that
$N_5^{N_4}=N_2, N_6^{N_4}=N_3$, and so
$K\simeq\mathbb{Z}_3\times((S_3\times S_3):\mathbb{Z}_2)\simeq C_W(w)$, as claimed.

\textbf{Torus 6}. In this case $w=w_1w_3w_{5}$ and $C_W(w)=~\langle w\rangle\times\langle w_2, w_{36} \rangle\simeq\mathbb{Z}_6 \times S_3$.
Using MAGMA we see that $[n, n_2]=[n, n_{36}]=1$. Therefore $n_2, n_{36}\in N$.
Let $\zeta$ be an element of $\overline{\F}_p$ such that $|\zeta|=2(q+1)$. Then it is clear that $\zeta^{(q+1)}=-1$.
Put $H_1=(1,1,1,1,-\zeta,-1)$, $N_1=H_1n$, $N_2=h_{36}n_2$ and $N_3=h_{2}n_{36}$. We claim that $K=\langle N_1, N_2, N_3\rangle$ is a complement for $T$ in $N$.
By Lemma~\ref{conjugation},
$$H^{n}=(\lambda_3^{-1}\lambda_4, \lambda_2, \lambda_1\lambda_3^{-1}\lambda_4, \lambda_4, \lambda_4\lambda_5^{-1}\lambda_6,\lambda_6).$$
Using this equality, we see that $h_2^{n}=h_2$, $h_{36}^{n}=h_{36}$ and hence $h_2, h_{36}\in T$. Furthermore, we have $H_1^{n}=(1,1,1,1,\zeta^{-1},-1)$. So $H_1^{\sigma n}=(1,1,1,1,\zeta^{-q},-1)$.
Since $\zeta^{q}=-\zeta^{-1}$, we infer that $H_1^{\sigma{n}}=(1,1,1,1,-\zeta,-1)=H_1$.
Calculations in MAGMA show that $(N_2)^2=1$, $(N_3)^2=1$, and $(N_2N_3)^3=1$.
It remains to verify that $N_1^6=1$ and $[N_1,N_2]=[N_1,N_3]=1$.
Applying MAGMA, we see that $n^6=h_5$, and therefore Lemma~\ref{conjugation} implies that $$(Hn)^6=(\lambda_4^2, \lambda_2^6, \lambda_4^4, \lambda_4^6, -\lambda_4^3\lambda_6^3, \lambda_6^6).$$ 
This equality yields $N_1^6=1$. By Lemma~\ref{conjugation},
$$H^{n_2}=(\lambda_1,\lambda_2^{-1}\lambda_4,\lambda_3,\lambda_4,\lambda_5,\lambda_6), H^{n_{36}}=(\lambda_1\lambda_2^{-1}, \lambda_2^{-1}, \lambda_2^{-2}\lambda_3,\lambda_2^{-3}\lambda_4,\lambda_2^{-2}\lambda_5, \lambda_2^{-1}\lambda_6).$$
It follows that $H_1^{-1}H_1^{n_{2}}=(1,1,1,1,1,1)=h_{36}^{-1}h_{36}^{n}$ and $H_1^{-1}H_1^{n_{36}}=(1,1,1,1,1,1)=h_{2}^{-1}h_{2}^{n}$. Therefore $[N_1,N_2]=[N_1,N_3]=1$ by Lemma~\ref{commutator}. Thus $\langle N_1, N_2, N_3 \rangle\simeq\mathbb{Z}_6\times S_3$, as claimed.

\textbf{Torus 9.} In this case $w=w_1w_2w_3w_5$ and $C_W(w_1w_2w_3w_5)=\langle w_1w_3\rangle\times \langle w_2,w_5, v\rangle\simeq
\mathbb{Z}_3\times D_8$, where
$v=w_1w_4w_{14}w_{29}, w_2^v=w_5$ and $D_8=(w_2\times w_5)\rtimes v.$
Let
$\xi$ be a primitive $(q^2-1)$th root of unity and $\lambda=\xi^{\frac{q-1}{2}}$. 
Observe that $\lambda^{q+1}=\xi^{(q^2-1)/2}=-1$.
Put
$$N_1=n_1n_3, N_2=H_2n_2, N_3=H_3n_5, N_4=h_1h_4n_1n_4n_{14}n_{29},$$
where $H_2=(-1,\lambda,1,-1,1,-1), H_3=(1,1,1,1,\lambda^{-1},-1)$.

Using MAGMA, we see that $[n,n_1n_3]=[n,n_2]=[n,n_5]=[n,N_4]=1$, whence
$N_1$, $n_2$, $n_5$, $N_4$ belong to $N$.
By Lemma~\ref{conjugation}, 
$$H^n=(\lambda_3^{-1}\lambda_4,\lambda_2^{-1}\lambda_4,\lambda_1\lambda_3^{-1}\lambda_4,
\lambda_4,\lambda_4\lambda_5^{-1}\lambda_6,\lambda_6).$$
Since $\lambda^{q+1}=-1$, we have
$H_2^{\sigma{n}}=(-1,-\lambda^{-1},1,-1,1,-1)^\sigma=H_2$ and $H_3^{\sigma{n}}=(1,1,1,1,-\lambda,-1)^\sigma=H_3$, 
Therefore, $H_2$ and $H_3$ belong to $T$.
We claim that $K=\langle N_1,N_2,N_3,N_4\rangle$ is a complement.

Calculations in MAGMA show that $N_1^3=[N_1,N_4]=1$. Now we prove that $[N_1,N_2]=[N_1,N_3]=1$.
Since $[n_1n_3,n_2]=[n_1n_3,n_5]=1$, it suffices to verify $H_2^{-1}H_2^{N_1}=H_3^{-1}H_3^{N_1}=1$.
By Lemma~\ref{conjugation}, $H^{N_1}=(\lambda_3^{-1}\lambda_4, \lambda_2, \lambda_1\lambda_3^{-1}\lambda_4,\lambda_4,
\lambda_5,\lambda_6)$ and hence $H_2^{N_1}=H_2$ and $H_3^{N_1}=H_3$, as required.

Now we verify that $\langle N_2,N_3,N_4 \rangle\simeq D_8$. Using MAGMA, we see that $N_4^2=1$.
By Lemma~\ref{conjugation}, $(Hn_2)^2=(\lambda_1^2,-\lambda_4,\lambda_3^2,\lambda_4^2,\lambda_5^2,\lambda_6^2)$
and $(Hn_5)^2=(\lambda_1^2,\lambda_2^2,\lambda_3^2,\lambda_4^2,-\lambda_4\lambda_6,\lambda_6^2)$.
Therefore $N_2^2=N_3^2=1$. It remains to verify the equation $N_2N_4=N_4N_3$. Note that $N_2N_4=H_2n_2N_4$. On the other hand, $N_4N_3=H_3^{N^4}N_4n_5$. By Lemma~\ref{conjugation}, 
$$H^{N_4}=(\lambda_1^{-1}\lambda_6, \lambda_5^{-1}\lambda_6^2,\lambda_3^{-1}\lambda_6^2,\lambda_3^{-1}\lambda_6^3, \lambda_2^{-1}\lambda_6^2,\lambda_6).$$
Therefore $H_3^{N_4}=(1,-\lambda,1,-1,1,-1)=H_2$. Using MAGMA, we see that $n_2N_4=N_4n_5$ and hence $N_2N_4=N_4N_3$.
Thus $K\simeq\langle N_1\rangle\times\langle N_2,N_3,N_4\rangle\simeq \mathbb{Z}_3\times D_8\simeq C_W(w)$.

\textbf{Torus 10.} In this case $w=w_1w_5w_3w_6$ and $C_W(w)=\langle w,w_2, w_{36},u,v\rangle
\simeq\mathbb{Z}_3\times S_3\times S_3$,
where $u=w_2w_{26}w_{28}w_{34}, v=w_2w_{24}w_{32}w_{33}$ and $\langle w_2,w_{36}\rangle\simeq\langle u,v\rangle\simeq S_3.$
Put
$$N_1=n, N_2=h_{36}n_2,N_3=h_2n_{36}, N_4=h_1h_6n_2n_{26}n_{28}n_{34}, N_5=h_1h_3h_6n_2n_{24}n_{32}n_{33}.$$
We claim that $K=\langle N_1,N_2,N_3,N_4,N_5\rangle$ is a complement.
Using MAGMA, we see that $[N_1,N_2]=[N_1,N_3]=[N_1,N_4]=[N_1,N_5]=1$ and $N_1^3=1$. Therefore, $N_1, N_2, N_3, N_4, N_5$ belong to $N$
and $K=\langle N_1\rangle\times\langle N_2,N_3,N_4,N_5\rangle$.

Computations in MAGMA show that $N_2^2=N_3^2=N_4^2=N_5^2=1, (N_2N_3)^3=1$ and  $(N_4N_5)^3=1$. Whence
$\langle N_2,N_3\rangle\simeq\langle N_4,N_5\rangle\simeq S_3.$ Finally, we have $[N_2,N_4]=[N_2,N_5]=[N_3,N_4]=[N_3,N_5]=1$. Thus $K\simeq C_W(w)$, as claimed.

\textbf{Torus 12.} In this case $w=w_1w_4w_3w_2$ and 
$C_W(w)=\langle w,w_{6}\rangle\simeq\mathbb{Z}_5\times\mathbb{Z}_2$.
Put
$$N_1=n_1n_4n_3n_2, N_2=h_2h_5n_6.$$
Using MAGMA, we see that $[N_1,N_2]=1$, and therefore $N_2$ belongs to $N$. Moreover, we have $N_1^5=N_2^2=1$, so the group $K=\langle N_1,N_2\rangle\simeq\mathbb{Z}_5\times\mathbb{Z}_2$ is a complement for $T$.

\textbf{Torus 13.} In this case $w=w_3w_2w_5w_4$ and
$C_W(w)=\langle w,w_{17}w_{18},w_{20}w_{21}\rangle\simeq\mathbb{Z}_6\times S_3$.
Let $$N_1=n_3n_2n_5n_4, N_2=h_3h_5n_{17}n_{18}, N_3=h_4h_6n_{20}n_{21}.$$
Using MAGMA, we see that $[N_1,N_2]=[N_1,N_3]=1$. Therefore, $N_2$ and $N_3$ belong to $N$.
Moreover, we have $N_1^6=1, N_2^2=N_3^2=(N_2N_3)^3=1.$ Thus 
$K=\langle N_1,N_2,N_3\rangle\simeq\mathbb{Z}_6\times S_3$ is a required complement for $T$.

\textbf{Torus 14.} In this case $w=w_3w_2w_4w_{14}$, $q\equiv1\pmod{4}$ and $C_W(w)\simeq SL_2(3):\mathbb{Z}_4$.
Moreover, we have $C_W(w)=\langle d,y,c \rangle$, where $d=w_6w_{15}w_{20}$, $y=w_4w_{11}w_{28}$ and $c=w_1w_2w_4w_6w_{31}w_{32}$. Note that $d^4=y^4=c^3=1$, $yd=dy$ and $d^3y^2c=c^2y$.
Using GAP, we see that these relations are determined $C_W(w)$ as abstract group generated by three elements.
Put $n=n_3n_2n_4n_{14}$, $D=h_6n_6n_{15}n_{20}$, $Y=h_4n_4n_{11}n_{28}$ and $C=h_1h_6n_1n_2n_4n_6n_{31}n_{32}$. Using MAGMA, we see that $[n,D]=[n,Y]=[n,C]=1$ and hence $D$, $Y$, $C$ are elements of the normalizer of $T$.
Let $\alpha$ be an element of $\overline{\F}_p$ such that $\alpha^2=-1$ and $H_1=(-1,-1,\alpha,1,\alpha,-1)$, $H_2=(-1,\alpha,1,-1,-\alpha,1)$. Put $N_1=H_1D$ and $N_2=H_2Y$. We claim that $K=\langle N_1, N_2, C\rangle$ is a complement for $T$ in $N$. It suffices to verify that $H_1$ and $H_2$ belong to $T$, and $N_1^4=N_2^4=C^3=1$, $N_1N_2=N_2N_1$ and $N_1^3N_2^2C=C^2N_2$.

By Lemma~\ref{conjugation}, we get 
$H^{n}=(\lambda_1, \lambda_3\lambda_4^{-1}\lambda_5,\lambda_1\lambda_3\lambda_4^{-1}\lambda_6,\lambda_3^2\lambda_4^{-1}\lambda_6,\lambda_2^{-1}\lambda_3\lambda_6,\lambda_6).$

Applying to $H_1$ and $H_2$, we 
obtain $H_1^n=H_1$ and $H_2^n=H_2$. Since $q\equiv1\pmod{4}$, we have $\alpha^q=\alpha$ and hence $H_1^{\sigma n}=H_1$, $H_2^{\sigma n}=H_2$. So $H_1$ and $H_2$ belong to $T$. Since $D^4=Y^4=h_2h_3$, Lemma~\ref{conjugation} implies that 
$$(HD)^4=(\lambda_1^4,-\lambda_1\lambda_2^2\lambda_3^2\lambda_5^{-2},-\lambda_1^3\lambda_2^2\lambda_3^2\lambda_5^{-2},\lambda_1^4\lambda_4^4\lambda_5^{-4},\lambda_1^4,\lambda_1^2)$$ and
$$(HY)^4=(\lambda_1^4,-\lambda_1^3\lambda_2^2\lambda_3^{-2}\lambda_5^2\lambda_6^{-2},-\lambda_1^3\lambda_2^{-2}\lambda_3^2\lambda_5^2\lambda_6^{-2},\lambda_1^4\lambda_5^4\lambda_6^{-4},\lambda_1^2\lambda_5^4\lambda_6^{-4},\lambda_1^2).$$ Now it is easy to see that $N_1^4=N_2^4=1$. Using MAGMA, we obtain $C^3=1$.

It follows from Lemma~\ref{commutator} that the equality $N_1N_2=N_2N_1$ is equivalent to $H_1^{-1}H_1^Y=H_2^{-1}H_2^D[D,Y]$. By Lemma~\ref{conjugation},
we have 
$$H^{-1}H^D=(1,\lambda_2^{-1}\lambda_3\lambda_6^{-1},\lambda_1\lambda_5^{-1},\lambda_1\lambda_2^{-1}\lambda_3\lambda_5^{-1}\lambda_6^{-1},\lambda_1\lambda_2^{-1}\lambda_3\lambda_5^{-1}\lambda_6^{-1},\lambda_1\lambda_6^{-2})$$
and
$$H^{-1}H^Y=(1,\lambda_1\lambda_4^{-1}\lambda_5\lambda_6^{-1},\lambda_1\lambda_4^{-1}\lambda_5\lambda_6^{-1},\lambda_1\lambda_2\lambda_3\lambda_4^{-2}\lambda_5\lambda_6^{-2},\lambda_1\lambda_6^{-2},\lambda_1\lambda_6^{-2}).$$

Therefore, we have $H_1^{-1}H_1^Y=(1,\alpha,\alpha,-1,-1,-1)$ and $H_2^{-1}H_2^D=(1,-\alpha,-\alpha,-1,-1,-1)$.
Since $[D,Y]=h_2h_3$, we infer that $N_1N_2=N_2N_1$.

It remains to verify the equality $N_1^3N_2^2C=C^2N_2$. Observe that $N_1^3N_2^2C=(H_1D)^3(H_2Y)^2C=H_1H_1^DH_1^{D^2}(H_2H_2^Y)^{D^3}D^3Y^2C$.
By Lemma~\ref{conjugation}, we have
$$(HD)^3=(\lambda_1^3,\lambda_2\lambda_3^2\lambda_5^{-1},\lambda_1^2\lambda_2\lambda_3^2\lambda_5^{-2}\lambda_6,\lambda_1^2\lambda_2^{-1}\lambda_3\lambda_4^3\lambda_5^{-3}\lambda_6,\lambda_1^2\lambda_2^{-1}\lambda_3\lambda_6,\lambda_1\lambda_6)D^3,$$
$$(HY)^2=(\lambda_1^2,\lambda_1\lambda_2^2\lambda_4^{-1}\lambda_5\lambda_6^{-1},\lambda_1\lambda_3^2\lambda_4^{-1}\lambda_5\lambda_6^{-1},\lambda_1\lambda_2\lambda_3\lambda_5\lambda_6^{-2},\lambda_1\lambda_5^2\lambda_6^{-2},\lambda_1)Y^2,$$
$$H^{D^3}=(\lambda_1,\lambda_1\lambda_2\lambda_5^{-1},\lambda_1\lambda_2\lambda_6^{-1},\lambda_1^2\lambda_2\lambda_3^{-1}\lambda_4\lambda_5^{-1}\lambda_6^{-1},\lambda_1^2\lambda_2\lambda_3^{-1}\lambda_6^{-1},\lambda_1\lambda_6^{-1}).$$
Therefore, $H_1H_1^DH_1^{D^2}=(-1,-\alpha,1,-1,\alpha,1)$ and $H_2H_2^Y=(1,\alpha,-\alpha,-1,1,-1)$.
Finally, $(H_2H_2^Y)^{D^3}=(1,\alpha,-\alpha,-1,1,-1)$. Then $N_1^3N_2^2C=(-1,1,-\alpha,1,\alpha,-1)D^3Y^2C$.
On the other hand, we have $C^2N_2=H_2^{C^2}C^2Y$. By Lemma~\ref{conjugation}, we obtain
$$H^{C^2}=(\lambda_1^{-1}\lambda_6,\lambda_1^{-1}\lambda_2^{-1}\lambda_4\lambda_5^{-1}\lambda_6,\lambda_1^{-1}\lambda_5^{-1}\lambda_6^2,\lambda_1^{-2}\lambda_2^{-1}\lambda_3\lambda_5^{-1}\lambda_6^2,\lambda_1^{-2}\lambda_3\lambda_5^{-1}\lambda_6,\lambda_1^{-1}).$$

Therefore, $C^2N_2=(-1,1,-\alpha,1,\alpha,-1)C^2Y$. Calculations in MAGMA show that $D^3Y^2C=C^2Y$ and hence 
$N_1^3N_2^2C=C^2N_2$. Thus $K$ is a complement, as claimed.

\textbf{Torus 15}. In this case, $w=w_1w_5w_3w_6w_2$ and $C_W(w)=\langle w\rangle\times\langle w_{24}w_{32}w_{33}, w_{26}w_{28}w_{34}\rangle\simeq\mathbb{Z}_6\times S_3$.
Let $\xi, \zeta$ be elements of $\overline{\F}_p$ such that $|\xi|=2(q^3-1)$ and $|\zeta|=2(q+1)$. 
Observe that $\xi^{q^3-1}=-1$, $\zeta^{q+1}=-1$.
Put $$H_1=(-1,\zeta,1,-1,-\xi^{q-1},-\xi^{q^2-1}),$$ $$H_2=(\xi^{q^2+q},\xi^{q^2+q+1},-\xi^{2q^2+q+1},-\xi^{2(q^2+q+1)},\xi^{(q+1)^2},\xi^{q^2+q}),$$
$$H_3=(\xi^{q+1},\xi^{q^2+q+1},\xi^{(q+1)^2},-\xi^{2(q^2+q+1)},\xi^{(q+1)^2},\xi^{q^2+q}).$$ 
For convenience,
we denote $n_{24}n_{32}n_{33}$ by $u$ and $n_{26}n_{28}n_{34}$ by $v$. Put $N_1=H_1n$, $N_2=H_2h_1h_3h_6u$
and $N_3=H_3h_1h_6v$. We claim that $K=\langle N_1, N_2, N_3 \rangle$ is a complement for $T$ in $N$.
Using MAGMA, we see that  $[h_1h_3h_6u,n]=[h_1h_6v,n]=1$, so $h_1h_3h_6u$, $h_1h_6v$ belong to $N$ by Lemma~\ref{normalizer}.
Now we verify that $H_1$, $H_2$ and $H_3$ belong to $T$.
By Lemma~\ref{conjugation},
$$H^{n}=(\lambda_3^{-1}\lambda_4,\lambda_2^{-1}\lambda_4,\lambda_1\lambda_3^{-1}\lambda_4,\lambda_4,\lambda_4\lambda_6^{-1},\lambda_5\lambda_6^{-1}).$$
Putting $H=H_1$, we obtain $H_1^{n}=(-1,-\zeta^{-1}, 1, -1, \xi^{1-q^2},\xi^{q-q^2})$. Whence $H_1^{\sigma{n}}=(-1,-\zeta^{-q},1,-1,\xi^{q-q^3},\xi^{q^2-q^3})$.
Observe that $\zeta^{-q}=-\zeta$ and $\xi^{q^3}=-\xi$. Therefore, $H_1^{\sigma{n}}=(-1,\zeta,1,-1,-\xi^{q-1},-\xi^{q^2-1})=H_1$. Furthermore,
$$H_2^{n}=(\xi^{q+1},-\xi^{q^2+q+1},\xi^{q^2+2q+1},-\xi^{2(q^2+q+1)},-\xi^{q^2+q+2},\xi^{q+1}).$$
Therefore, $H_2^{\sigma{n}}=(\xi^{q^2+q},-\xi^{q^3+q^2+q},\xi^{q^3+2q^2+q},-\xi^{2(q^3+q^2+q)},-\xi^{q^3+q^2+2q},\xi^{q^2+q})$. Since $\xi^{q^3}=-\xi$, we have $H_2^{\sigma{n}}=(\xi^{q^2+q},\xi^{1+q^2+q},-\xi^{1+2q^2+q},-\xi^{2(1+q^2+q)},\xi^{1+q^2+2q},\xi^{q^2+q})=H_2$. Finally,
$H_3^n=(-\xi^{q^2+1},-\xi^{q^2+q+1},-\xi^{q^2+q+2},-\xi^{2(q^2+q+1)},-\xi^{q^2+q+2},\xi^{q+1})$, so $H_3^{\sigma{n}}=(-\xi^{q^3+q},-\xi^{q^3+q^2+q},-\xi^{q^3+q^2+2q},-\xi^{2(q^3+q^2+q)},-\xi^{q^3+q^2+2q},\xi^{q^2+q})$.
Since $\xi^{q^3}=-\xi$ and $\xi^{2q^3}=\xi^2$, we have $H_3^{\sigma{n}}=(\xi^{1+q},\xi^{1+q^2+q},\xi^{1+q^2+2q},-\xi^{2(1+q^2+q)},\xi^{1+q^2+2q},\xi^{q^2+q})=H_3$.
Since $n^6=h_2$, Lemma~\ref{conjugation} implies that $(H_1n)^6=(\lambda_4^2,-\lambda_4^3,\lambda_4^4,\lambda_4^6,\lambda_4^4,\lambda_4^2)$, whence $|N_1|=6$.
Now we prove that $N_2$ and $N_3$ are involutions. By Lemma~\ref{conjugation}, $(Hu)^2=(\lambda_1\lambda_6^{-1},-\lambda_2^2\lambda_4^{-1},\lambda_3\lambda_4^{-1}\lambda_5\lambda_6^{-1},1,\lambda_1^{-1}\lambda_3\lambda_4^{-1}\lambda_5,\lambda_1^{-1}\lambda_6)$. Putting $H=H_2h_1h_3h_6$, we see that
 $N_2^2=1$. Similarly, we have  
$$(Hv)^2=(\lambda_1\lambda_5^{-1}\lambda_6,-\lambda_2^2\lambda_4^{-1},\lambda_3\lambda_5^{-1},1,\lambda_3^{-1}\lambda_5,\lambda_1\lambda_3^{-1}\lambda_6).$$ 
Applying this to $H_3h_1h_6$ we obtain $N_3^2=1$.

Now we prove that $N_1$ commutes with $N_2$ and $N_3$. By Lemma~\ref{commutator}, this is equivalent to $H_1^{-1}H_1^{N_2}=H_2^{-1}H_2^{N_1}$ and $H_1^{-1}H_1^{N_3}=H_3^{-1}H_3^{N_1}$, respectively.
Using the equations for $H_2^n$ and $H_3^n$, we have $$H_2^{-1}H_2^{N_1}=(\xi^{1-q^2},-1,-\xi^{q-q^2},1,-\xi^{1-q},\xi^{1-q^2})=H_1^{-1}H_1^{N_2},$$ 
$$H_3^{-1}H_3^{N_1}=(-\xi^{q^2-q},-1,-\xi^{1-q},1,-\xi^{1-q},\xi^{1-q^2})=H_1^{-1}H_1^{N_3},$$ as required.

It remains to prove that $(N_2N_3)^3=1$. Observe that 
\begin{multline*}
N_2N_3=H_2h_1h_3h_6uH_3h_1h_6v=H_2h_1h_3h_6(H_3h_1h_6)^{u}uv=\\
=(\ast,\xi^{q^2+q+1},\ast,-\xi^{2(q^2+q+1)},\ast,\ast)(\ast,-\xi^{-(q^2+q+1)},\ast,-\xi^{-2(q^2+q+1)},\ast,\ast)uv
=(\ast,-1,\ast,1,\ast,\ast)uv \end{multline*}
By Lemma~\ref{conjugation},
$$(Huv)^3=(\lambda_4,-\lambda_2^3,\lambda_4^2,\lambda_4^3,\lambda_4^2,\lambda_4).$$ 
Thus $(N_2N_3)^3=1$ and $K$ is a required complement.

\textbf{Torus 17}. In this case $w=w_1w_4w_5w_3w_{36}$ and $C_W(w)=~\langle w\rangle\simeq\mathbb{Z}_{10}$.
Let $\xi$ be an element of $\overline{\F}_p$ such that $|\xi|=(q+1)(q^5-1)$. Let $\zeta=\xi^{(q-1)/2}$. Observe that $\xi^{q(q^5-1)}\xi^{(q^5-1)}=1$ and hence $\zeta^{q^6-q}=\zeta^{1-q^5}$.
Put $$H_1=(\zeta^{q^6+q^3-q},\zeta^{-q^5+1},\zeta^{-q^5+q^4+q^3+1},\zeta^{-q^5+q^4+q^3+q^2+1},\zeta^{q^4+q^3+q^2+1},\zeta^{q^4+q^3+q^2+q+1}).$$
Now we show that $H_1\in T$. By Lemma~\ref{conjugation},
$$H^n=(\lambda_2^{-1}\lambda_3^{-1}\lambda_4,\lambda_2^{-1},\lambda_1\lambda_2^{-2}\lambda_3^{-1}\lambda_4,\lambda_1\lambda_2^{-2}\lambda_3^{-1}\lambda_4\lambda_5^{-1}\lambda_6,\lambda_2^{-2}\lambda_4\lambda_5^{-1}\lambda_6,
\lambda_2^{-1}\lambda_6).$$ 
Putting $H=H_1$, we see that $$H_1^{n}=(\zeta^{q^5+q^2-1},\zeta^{q^5-1},\zeta^{q^5+q^3+q^2-1},\zeta^{q^5+q^3+q^2+q-1},\zeta^{q^5+q^4+q^3+q^2+q-1},\zeta^{q^5+q^4+q^3+q^2+q}).$$
Therefore, 
\begin{multline*}
H_1^{\sigma{n}}=(\zeta^{q^6+q^3-q},\zeta^{q^6-q},\zeta^{q^6+q^4+q^3-q},\zeta^{q^6+q^4+q^3+q^2-q},\zeta^{q^6+q^5+q^4+q^3+q^2-q},\zeta^{q^6+q^5+q^4+q^3+q^2})=\\
=(\zeta^{q^6+q^3-q},\zeta^{1-q^5},\zeta^{-q^5+q^4+q^3+1},\zeta^{-q^5+q^4+q^3+q^2+1},\zeta^{q^4+q^3+q^2+1},\zeta^{q^4+q^3+q^2+q+1})=H_1.
\end{multline*}
Thus $H_1$ belongs to $T$.

Using MAGMA, we see that $n^{10}=h_1h_4h_6$, and therefore Lemma~\ref{conjugation} implies that  $(Hn)^{10}=(-\lambda_2^{-1}\lambda_6^2,1,\lambda_2^{-2}\lambda_6^4,-\lambda_2^{-3}\lambda_6^6,\lambda_2^{-4}\lambda_6^8,-\lambda_2^{-5}\lambda_6^{10})$. 

Putting $H=H_1$, we have $\lambda_2^{-1}\lambda_6^2=\zeta^{q^5-1}\zeta^{2q^4+2q^3+2q^2+2q+2}=\zeta^{q^5+2q^4+2q^3+2q^2+2q+1}=\zeta^{(q+1)(q^4+q^3+q^2+q+1)}=\xi^{(q^5-1)(q+1)/2}=-1$. Thus $(H_1n)^{10}=1$ and $\langle H_1n \rangle$ is a complement for $T$ in $N$.

\textbf{Torus 18.} In this case $w=w_1w_4w_6w_3w_5$ and $C_W(w)=\langle w,w_{36}\rangle\simeq\mathbb{Z}_6\times\mathbb{Z}_2$.
Let $\xi$ and $\zeta$ be elements of $\overline{\F}_p$ such that $\xi^{q+1}=-1$ and $\zeta^{q-1}=-1$. Put $H_1=(\xi, -1, -1, \xi^{-1},-1,\xi)$ and $H_2=(\zeta,-\zeta^2,-\zeta^2,\zeta^3,-\zeta^2,\zeta)$.
Now we verify that $H_1, H_2\in T$. By Lemma~\ref{conjugation},
$$H^{n}=(\lambda_3^{-1}\lambda_4,\lambda_2,\lambda_1\lambda_3^{-1}\lambda_4,\lambda_1\lambda_2\lambda_3^{-1}\lambda_4\lambda_5^{-1}\lambda_6,\lambda_4\lambda_5^{-1}\lambda_6,\lambda_4\lambda_5^{-1}).$$
So $H_1^{\sigma{n}}=(-\xi^{-q}, (-1)^{q},(-1)^{q}, (-\xi)^q,(-1)^q,(-\xi)^{-q})$. Since $\xi^{q+1}=-1$, we have $-\xi^{-q}=\xi$ and $(-\xi)^q=\xi^{-1}$. Therefore, $H_1^{\sigma{n}}=H_1$ and hence $H_1\in T$.
Now $H_2^{\sigma{n}}=(-\zeta^q,-\zeta^{2q},-\zeta^{2q},-\zeta^{3q},-\zeta^{2q},-\zeta^q)$. Since $\zeta^q=-\zeta$, we have $H_2^{\sigma{n}}=(\zeta,-\zeta^2,-\zeta^2,\zeta^3,-\zeta^2,\zeta)=H_2$.

Put $N_1=H_1n_1n_4n_6n_3n_5$ and $N_2=H_2n_{36}$. We claim that $K=\langle N_1, N_2\rangle$ is a complement for $T$ in $N$.
It suffices to show that $|N_1|=6$, $|N_2|=2$ and $[N_1,N_2]=1$.
Using MAGMA, we see that $n^6=h_1h_4h_6$, and therefore Lemma~\ref{conjugation} implies that $(Hn)^6=(-\lambda_2^3,\lambda_2^6,\lambda_2^6,-\lambda_2^9,\lambda_2^6,-\lambda_2^3)$.
Therefore, we obtain $N_1^6=(H_1n)^6=1$. Similarly, we have
$(Hn_{36})^2=(-\lambda_1^2\lambda_2^{-1},1,\lambda_2^{-2}\lambda_3^2,-\lambda_2^{-3}\lambda_4^2,\lambda_2^{-2}\lambda_5^2,-\lambda_2^{-1}\lambda_6^2)$, so $(H_2n_{36})^2=1$. Note that $[n,n_{36}]=1$, so it remains to prove that
$H_2^{-1}H_2^{n}=H_1^{-1}H_1^{n_{36}}$ by Lemma~\ref{commutator}.
Using equations for $H^{n}$ and $H^{n_{36}}$, we obtain $$H_1^{-1}H_1^{n_{36}}=(\xi^{-1}, -1, -1, \xi,-1,\xi^{-1})(-\xi, -1, -1, -\xi^{-1},-1,-\xi)=(-1,1,1,-1,1,-1).$$
and $$H_2^{-1}H_2^{n}=(\zeta^{-1},-\zeta^{-2},-\zeta^{-2},\zeta^{-3},-\zeta^{-2},\zeta^{-1})(-\zeta,-\zeta^{2},-\zeta^{2},-\zeta^{3},-\zeta^{2},-\zeta)=(-1,1,1,-1,1,-1).$$
Thus $[N_1,N_2]=1$ and $\langle N_1, N_2\rangle$ is a complement for $T$, as claimed.

\textbf{Tori 19, 23, 24.} In these cases $C_W(w)$ is cyclic and generated by $w$.
Depending on torus, we have $w=w_2w_5w_3w_4w_6$, $w_1w_4w_6w_3w_2w_5$ or $w_1w_4w_{14}w_3w_2w_6$. Put $N_1=n_2n_5n_3n_4n_6$, $n_1n_4n_6n_3n_2n_5$ or $n_1n_4n_{14}n_3n_2n_6$, respectively. Calculations in MAGMA show that $|N_1|=|w|$ in each case and so $\langle N_1\rangle$ is a complement.

\textbf{Torus 20.} In this case $w=w_{20}w_5w_4w_3w_2$ and $C_W(w)=~\langle w\rangle\simeq\mathbb{Z}_{12}$.
Put $n=n_{20}n_5n_4n_3n_2$. Let $\xi$ be an element of $\overline{\F}_p$ such that $|\xi|=|T|=(q-1)(q^2+1)(q^3+1)$ and $\zeta=\xi^{(q-1)/2}$.
$$H_1=(-1,-\zeta^{q},-\zeta^{-q^4},-\zeta^{-q^4-q^3},\zeta^{-q^4-q^3-q^2-1},\zeta^{-q^3-1}).$$
By Lemma~\ref{conjugation}, 
$$H^n=(\lambda_1, \lambda_1\lambda_3^{-1}\lambda_4\lambda_6^{-1}, \lambda_1\lambda_3^{-1}\lambda_4,\lambda_1^2\lambda_3^{-2}\lambda_4\lambda_5\lambda_6^{-1},\lambda_1^2\lambda_3^{-2}\lambda_4,\lambda_1\lambda_2\lambda_3^{-1}).$$
Therefore, we obtain $H_1^n=(-1,-\zeta,-\zeta^{-q^3},-\zeta^{-q^3-q^2},-\zeta^{q^4-q^3},-\zeta^{q^4+q})$. Hence $H_1^{\sigma n} =(-1,-\zeta^{q},-\zeta^{-q^4},-\zeta^{-q^4-q^3},-\zeta^{q^5-q^4},-\zeta^{q^5+q^2}).$
Observe that $|\xi|=2(q^5+q^3+q^2+1)(q-1)/2$ and hence $\zeta^{q^5+q^3+q^2+1}=-1$. Then $\zeta^{q^5}=-\zeta^{-q^3-q^2-1}$. So $-\zeta^{q^5-q^4}=\zeta^{-q^4-q^3-q^2-1}$
and $-\zeta^{q^5+q^2}=\zeta^{-q^3-q^2-1+q^2}$. 
Thus $H_1^{\sigma n}=H_1$.
Using MAGMA, we have $n^{12}=h_2h_3$, and therefore $(Hn)^{12}=(\lambda_1^{12},-\lambda_1^9,-\lambda_1^{15},\lambda_1^{18},\lambda_1^{12},\lambda_1^6)$.
Thus $(H_1n)^{12}=1$ and $\langle H_1n\rangle$ is a complement for $T$ in $N$.

\textbf{Torus 21.} In this case $w=w_1w_5w_2w_3w_6w_{36}$ and $C_W(w)\simeq(((\mathbb{Z}_3\times\mathbb{Z}_3):\mathbb{Z}_3):Q_8):\mathbb{Z}_3$. Moreover,
we have $C_W(w)=\langle u,v\rangle$, where $u=w_1w_2w_5w_{23}w_{26}w_{31}, v=w_1w_2w_6w_8w_{10}w_{29}$
and
$$
\langle u,v\rangle\simeq\langle a,b~|~a^{12}=b^6=a^8ba^{-8}b^{-1}=(a^6b^{-1})^3=a^6b^2a^6b^{-2}=ba^8(a^{-1}b)^2a^{-1}=1\rangle
$$
Put
$N_1=h_1h_2h_5\cdot n_1n_2n_5n_{23}n_{26}n_{31}, N_2=h_1h_5\cdot n_1n_2n_6n_8n_{10}n_{29},$
and $N=n$. Using MAGMA, we see that $[N,N_1]=[N,N_2]=1$. So $N_1$ and $N_2$ belong to $N$ by Lemma~\ref{normalizer}.
It is easy to verify the relations of $C_W(w)$ for $N_1$ and $N_2$ using MAGMA, but we provide arguments that do not involve computations.

The relations of $C_W(w)$ are valid for $u,v$. Substituting $a=N_1$ and $b=N_2$, 
we obtain each relation up to some element $h\in T$. Since $N_1\in\mathcal{T}$ and $N_2\in\mathcal{T}$, we infer that every such $h\in\mathcal{T}\cap{T}=\mathcal{H}\cap{T}$. Since $T\simeq(q^2+q+1)^3$, 
every element of $T$ has odd order. On the other hand, $\mathcal{H}$ is an elementary abelain 2-group and hence
$\mathcal{H}\cap{T}=1$. Thus all relations of $C_W(w)$ hold true for $K=\langle N_1, N_2\rangle$,
and therefore $K$ is a required complement.

\textbf{Torus 22.} In this case $w=w_1w_4w_6w_3w_5w_{36}$ and $C_W(w)=\langle w,w_{36},w_{24}\rangle\simeq\mathbb{Z}_6\times S_3$. By Lemma~\ref{conjugation},
$$H^n=
(\lambda_2^{-1}\lambda_3^{-1}\lambda_4,\lambda_2^{-1},\lambda_1\lambda_2^{-2}\lambda_3^{-1}\lambda_4,
\lambda_1\lambda_2^{-2}\lambda_3^{-1}\lambda_4\lambda_5^{-1}\lambda_6,
\lambda_2^{-2}\lambda_4\lambda_5^{-1}\lambda_6,\lambda_2^{-1}\lambda_4\lambda_5^{-1}).$$
Using MAGMA, we see that $[n,n_{36}]=1$ and $[n,n_{24}]=h_2h_3h_5$.
Let ${h}=(\alpha^2,\alpha,\alpha,1,\alpha,\alpha^2)$, where
$\alpha\in\overline{\F}_p$ with $\alpha^{q+1}=-1$. Then ${h}^{\sigma{n}}=(\alpha^{-2},\alpha^{-1},\alpha^{-1},1,\alpha^{-1},\alpha^{-2})^\sigma=
(\alpha^{2},-\alpha,-\alpha,1,-\alpha,\alpha^{2})=h_2h_3h_5\cdot h$.
Hence, Lemma~\ref{normalizer} implies that ${h}n_{24}$ belongs to $N$.
Let
$\xi$ be a primitive $(2q^3+2)$th root of unity, $\lambda=-\xi^2$ and $\alpha=\xi^{-q^2+q-1}$. 
Observe that $\lambda^{-q^3}=-\xi^{-2q^3}=-\xi^2=\lambda$.
Put $$N_1=H_1n,N_2=H_2{h}n_{24},N_3=H_3n_{36},$$
where $H_1=(\lambda,1,\lambda^{q+1},\lambda^{-q^2+q+1},\lambda^{q+1},\lambda),
H_2=(-\alpha^{-2},1,1,-\alpha^2,1,-\alpha^{-2}),
H_3=h_2h_3h_5.$ We claim that $K=\langle N_1,N_2,N_3\rangle$ is a complement.
First, we have $H_1^{\sigma{n}}=
(\lambda^{-q^2},1,\lambda^{-q^2+1},\lambda^{-q^2-q+1},\lambda^{-q^2+1},\lambda^{-q^2})^q=
(\lambda,1,\lambda^{1+q},\lambda^{1-q^2+q},\lambda^{1+q},\lambda)=H_1$.
Moreover, $H_{2}^{\sigma{n}}=(-\alpha^2,1,1,-\alpha^{-2},1,-\alpha^2)^q=
(-\alpha^{-2},1,1,-\alpha^2,1,-\alpha^{-2})=H_{2}$ and $H_3^{\sigma{n}}=H_3^q=H_3$.
Hence, $H_1,H_2,H_3\in T$.
Lemma~\ref{conjugation} implies that 
$$H^{n_{24}}=(\lambda_1,\lambda_1\lambda_2\lambda_4^{-1}\lambda_6,\lambda_1\lambda_3\lambda_4^{-1}\lambda_6,
\lambda_1^2\lambda_4^{-1}\lambda_6^2,\lambda_1\lambda_4^{-1}\lambda_5\lambda_6,\lambda_6),$$
$$(Hn_{24})^2=(\lambda_1^2,-\lambda_1\lambda_2^2\lambda_4^{-1}\lambda_6,-\lambda_1\lambda_3^2\lambda_4^{-1}\lambda_6,
\lambda_1^2\lambda_6^2,-\lambda_1\lambda_4^{-1}\lambda_5^2\lambda_6,\lambda_6^2).$$
Therefore, $N_{2}^2=1.$ Using MAGMA, we see that $N_3^2=1$. Observe that
\begin{multline*}N_{2}N_{3}=H_{2}{h}n_{24}H_{3}n_{36}=H_{2}{h}H_{3}^{n_{24}}n_{24}n_{36}=\\
=(-1,\alpha,\alpha,-\alpha^2,\alpha,-1)(1,-1,-1,1,-1,1)n_{24}n_{36}=\\
=(-1,-\alpha,-\alpha,-\alpha^2,-\alpha,-1)n_{24}n_{36}.
\end{multline*}
Calculations in MAGMA show that $(n_{24}n_{36})^3=1$, and therefore
Lemma~\ref{conjugation} yields 
$$(Hn_{24}n_{36})^3=(\lambda_1^2\lambda_2^{-2}\lambda_4\lambda_6^{-1},1,
\lambda_2^{-3}\lambda_3^3,\lambda_1\lambda_2^{-4}\lambda_4^2\lambda_6,\lambda_2^{-3}\lambda_5^3,
\lambda_1^{-1}\lambda_2^{-2}\lambda_4\lambda_6^2).$$ Hence,
$(N_{2}N_{3})^3=1$ and $\langle N_{2},N_{3}\rangle\simeq S_3$.

Furthermore, by Lemma~\ref{commutator} we have
$N_1N_3=N_3N_1$ is equivalent to $H_{3}^{-1}H_{3}^{n}=H_1^{-1}H_1^{n_{36}}$.
Lemma~\ref{conjugation} implies that
$H^{n_{36}}=(\lambda_1\lambda_2^{-1},\lambda_2^{-1},\lambda_2^{-2}\lambda_3,
\lambda_2^{-3}\lambda_4,\lambda_2^{-2}\lambda_5,\lambda_2^{-1}\lambda_6).$
Then $H_1^{n_{36}}=H_1$ and $H_{3}^{n}=H_{3}$. Therefore $N_1N_3=N_3N_1.$
By Lemma~\ref{commutator}, we know that
$$N_1N_2=N_2N_1\text{ is equivalent to } H_1^{-1}H_1^{n_{24}}\cdot [n_{24},n]=(H_2{h})^{-1}(H_2{h})^n.$$
Lemma~\ref{conjugation} implies that
$H^{-1}H^{n_{24}}=(1,\lambda_1\lambda_4^{-1}\lambda_6,\lambda_1\lambda_4^{-1}\lambda_6,
\lambda_1^2\lambda_4^{-2}\lambda_6^2,\lambda_1\lambda_4^{-1}\lambda_6,1)$.
Putting $H=H_1$ and using $\lambda^{q^2-q+1}=-\xi^{2(q^2-q+1)}=-\alpha^{-2}$,
we have $H_1^{-1}H_1^{n_{24}}=(1,-\alpha^{-2},-\alpha^{-2},\alpha^{-4},-\alpha^{-2},1)$.
On the other hand,
$H_2{h}=
(-1,\alpha,\alpha,-\alpha^2,\alpha,-1),$ so
$(H_2{h})^n=(-1,\alpha^{-1},\alpha^{-1},-\alpha^{-2},\alpha^{-1},-1).$ Therefore,
$$(H_2{h})^{-1}(H_2{h})^n=
(1,\alpha^{-2},\alpha^{-2},\alpha^{-4},\alpha^{-2},1).$$
Since $[n,n_{24}]=h_2h_3h_5$, we get
$N_1N_2=N_2N_1$.

Finally, using MAGMA, we see that $(Hn)^6=1$. Thus
$N_1^6=1$ and $K\simeq\langle N_1\rangle\times\langle N_2,N_3\rangle\simeq\mathbb{Z}_6\times S_3$, as claimed.

\textbf{Torus 25.} In this case $w=w_1w_4w_{14}w_3w_2w_{31}$ and $C_W(w)=\langle w^2 \rangle\times\langle i,j,c \rangle $, where $i=w_3w_6w_{19}w_{26}$, $j=w_3w_6w_{14}w_{30}$, $c=w_1w_4w_6w_{13}w_{20}w_{34}$ and $\langle i,j,c\rangle\simeq SL_2(3)$. Observe that $|T|=(q^2-q+1)(q^4+q^2+1)$ is odd.

Put $N_1=n^2$, $N_2=h_1h_2h_5n_3n_6n_{19}n_{26}$, $N_3=h_2h_3h_4h_5n_3n_6n_{14}n_{30}$, $N_4=h_1h_2h_4h_6n_1n_4n_6n_{13}n_{20}n_{34}$.
We claim that $K=\langle N_1, N_2, N_3, N_4 \rangle$ is a complement for $T$ in $N$.
Using MAGMA, we see that $[n,N_2]=[n,N_2]=[n,N_3]=1$ and hence $N_2, N_3, N_4\in N$ by Lemma~\ref{normalizer}.
Clearly the image of $\langle N_1, N_2, N_3, N_4 \rangle$ in $W$ is $C_W(w)$. 
Now we argue as in the case of Torus~21. The group $C_W(w)$ is defined by some relations on $w^2$, $i$, $j$ and $c$.
These relations hold true for $N_1$, $N_2$, $N_3$ and $N_4$ up to some elements $h\in T\cap\mathcal{H}$.
However, since the order of $T$ is odd, we have $h=1$ for each relation. 
Thus $K\simeq C_W(w)$ and $K$ is a required complement.


\begin{table}
\caption{The maximal tori of simply connected group $E_6(q)$\label{table}}
{\small\centering
\begin{tabular}{|l|l|c|l|l|l|c|}
 \hline
 \No  & Representative $w$ & $|w|$ & $|C_W(w)|$ & Structure of $C_W(w)$ & Cyclic structure of $T$ & Splits\\ \hline
  1  & $1$ & 1 & 51840 & $O_5(3):Z_2$ & $(q-1)^6$ & --\\
  2  & $w_1$ & 2 & 1440 & $S_2\times S_6$ & $(q-1)^4\times(q^2-1)$ & --\\
  3  & $w_1w_2$ & 2 & 192 & $D_8\times S_4$ & $(q-1)^2\times(q^2-1)^2$ & --\\
  4  & $w_3w_1$ & 3 & 216 & $Z_3\times((S_3\times S_3):Z_2)$ & $(q-1)^3\times(q^3-1)$& +\\
  5  & $w_2w_3w_5$ & 2 & 96 & $Z_2\times Z_2\times S_4$ & $(q^2-1)^3$& --\\
  6  & $w_1w_3w_5$ & 6 & 36 &$Z_6\times S_3$ & $(q-1)\times(q^2-1)\times(q^3-1)$&+\\
  7 & $w_1w_3w_4$ & 4 & 32 &$Z_4\times D_8$ & $(q-1)^2\times(q^4-1)$& --\\
  8 & $w_1w_4w_6w_{36}$ & 2 & 1152 &$Z_2:(((A_4\times A_4):Z_2):Z_2)$  & $(q+1)^2\times(q^2-1)^2$ & --\\
  9 & $w_1w_2w_3w_5$ & 6 & 24 &$Z_3\times D_8$ & $(q^2-1)\times(q+1)(q^3-1)$& +\\
  10 & $w_1w_5w_3w_6$ & 3 &  108 &$Z_3\times S_3\times S_3$ & $(q-1)\times(q^2+q+1)\times(q^3-1)$& +\\
  11 & $w_1w_4w_6w_3$ & 4 &  16 &$Z_4\times Z_2\times Z_2$ & $(q^2-1)\times(q^4-1)$ & -- \\
  12 & $w_1w_4w_3w_2$ & 5 & 10 &$ Z_2\times Z_5$ & $(q-1)\times(q^5-1)$ & + \\
  13 & $w_3w_2w_5w_4$ & 6 & 36 &$ Z_6\times S_3$ & $(q^2-1)\times(q-1)(q^3+1)$& +\\
  14 & $w_3w_2w_4w_{14}$ & 4 & 96 &$ SL_2(3): Z_4$ & $(q-1)(q^2+1)^2$& + $(q\not\equiv3(4))$\\
  & & & & & & -- $(q\equiv3(4))$\\
  15 & $w_1w_5w_3w_6w_2$ & 6 &  36 &$ Z_6\times S_3$ & $(q^2+q+1)\times(q+1)(q^3-1)$& +\\
  16 & $w_1w_4w_6w_3w_{36}$ & 4 & 96 &$ Z_4\times S_4$ & $(q+1)^2\times(q^4-1)$& --\\
  17 & $w_1w_4w_5w_3w_{36}$ &  10 & 10 &$ Z_{10}$ & $(q+1)(q^5-1)$ & + \\
  18 & $w_1w_4w_6w_3w_5$ & 6  & 12 &$ Z_6\times Z_2$ & $(q^2+q+1)\times(q-1)(q^3+1)$& +\\
  19 & $w_2w_5w_3w_4w_6$ & 8  & 8 &$ Z_8$ & $(q^2-1)(q^4+1)$ & +\\
  20 & $w_{20}w_5w_4w_3w_2$ & 12 & 12 &$ Z_{12}$ & $(q-1)(q^2+1)(q^3+1)$ & + \\
  21 & $w_1w_5w_2w_3w_6w_{36}$ & 3 &  648 &$(((Z_3\times Z_3):Z_3):Q_8):Z_3$
    & $(q^2+q+1)^3$& +\\
  22 & $w_1w_4w_6w_3w_5w_{36}$ & 6 & 36 &$ Z_6\times S_3$ & $(q+1)\times(q^5+q^4+q^3+q^2+q+1)$& +\\
  23 & $w_1w_4w_6w_3w_2w_5$ & 12 &  12 &$ Z_{12}$ & $(q^2+q+1)(q^4-q^2+1)$& +\\
  24 & $w_1w_4w_{14}w_3w_2w_6$ & 9 &  9 &$ Z_9$ & $(q^6+q^3+1)$& +\\
  25 & $w_1w_4w_{14}w_3w_2w_{31}$ & 6 & 72 &$ Z_3\times SL_2(3)$ & $(q^2-q+1)\times(q^4+q^2+1)$& +\\
  \hline
\end{tabular}}
\end{table}


\Addresses

\begin{thebibliography}{99}

\bibitem{AdamsHe}
J.\,Adams, X.\,He, {\it Lifting of elements of Weyl groups}, J. Algebra. V.485 (2017), 142--165.

\bibitem{ButGre}
A.A.\,Buturlakin, M.A.\,Grechkoseeva, {\it The cyclic structure of maximal tori of the finite classical groups},  Algebra Logic, V.46:2 (2007), 73--89.

\bibitem{Car}
R.W.\,Carter, {\it Finite groups of Lie type, Conjugacy classes and complex
characters}, John Wiley and Sons, 1985.

\bibitem{DerF}
D.I.\,Deriziotis, A.P.\,Fakiolas, {\it The maximal tori in the finite Chevalley groups of type $E_6,E_7$ and $E_8$}, Comm. Algebra, V.19:3 (1991) 889--903.

\bibitem{Galt1}
A.A.\,Gal$'$t, {\it On the splitting of the normalizer of a maximal torus in the exceptional linear algebraic groups}, Izv. Math.,  V.81:2 (2017), 269--285.

\bibitem{Galt2}
A.A.\,Galt, {\it On splitting of the normalizer of a maximal torus in orthogonal groups}, J. Algebra Appl., V.16:9 (2017) 1750174 (23 pages).

\bibitem{Galt3}
A.A.\,Galt, {\it On splitting of the normalizer of a maximal torus in linear groups}, J. Algebra Appl., V.14:7 (2015) 1550114 (20 pages).

\bibitem{Galt4}
A.A.\,Gal$'$t, {\it On the splitting of the normalizer of a maximal torus in symplectic groups}, 
Izv. Math., V.78:3, (2014), 443--458.

\bibitem{GorLySol}
Gorenstein D., Lyons R., Solomon R., {\it The classification of the finite simple groups}. Number 3. Part I. Chapter A. Almost simple $K$-groups. Mathematical Surveys and Monographs, {\bfseries 40}, N.3, American Mathematical Society, Providence, RI, 1998.

\bibitem{Atlas} J.H.\.Conway, R.T.\.Curtis, S.P.\.Norton, R.A.\.Parker, R.A.\.Wilson,
{\it Atlas of Finite Groups}, Clarendon Press, Oxford, 1985.

\bibitem{Vavilov} N.A.\,Vavilov, {\it Do it yourself structure constants for Lie algebras of types $E_l$}, 
J. Math. Sci. (N.Y.), V.120:4 (2004), 1513--1548.

\bibitem{MC} http://magma.maths.usyd.edu.au/calc/

\bibitem{MAGMA} W.\,Bosma, J.\,Cannon, and C.\,Playoust, The Magma algebra system. I. The user language, J. Symbolic Comput., V.24 (1997), 235--265.

\bibitem{GAP} The GAP Group, GAP -- Groups, Algorithms, and Programming, Version 4.9.1; 2018. (http://www.gap-system.org) 

\bibitem{Tits}
J.\,Tits, {\it Normalisateurs de tores I. Groupes de Coxeter \'{E}tendus}, J. Algebra, V.4 (1966), 96--116.
\end{thebibliography}
\end{document}